\documentclass{article} 
\usepackage[latin1]{inputenc}
\usepackage[english]{babel}
\usepackage{geometry}
\usepackage{amsmath}
\usepackage{amssymb}
\usepackage{graphicx}        
\usepackage{epsfig}
\usepackage{mathrsfs,subfigure}
\usepackage{color}
\graphicspath{{FIGURES/}}
\usepackage{amssymb,amsmath}
\usepackage{amsfonts}
\usepackage{latexsym}
\usepackage{float}
\usepackage{subfigure}
\usepackage{placeins}
\usepackage{comment}
\usepackage{xspace}
\usepackage{color}
\usepackage{bm}
\usepackage{authblk}
\usepackage{amsthm}

\usepackage{graphicx}
\graphicspath{{}}
\usepackage{algorithm}
\usepackage{algorithmicx,algpseudocode}
\usepackage{todonotes}
\usepackage{enumerate}

\newtheorem{assumption}{Assumption}

\def\xE{{\bf x}_E}
\def\x{{\bf x}}

\def\o{\omega}

\newcommand{\norm}[1]{\|#1\|}
\newcommand{\seminorm}[1]{\left|#1\right|}
\newcommand{\abs}[1]{\left|#1\right|}
\newcommand{\normA}[1]{\|#1\|_{a}}
\newcommand{\dx}[1]{\textrm{\, d}#1}

\def\CN{{\mathcal N}}

\def\CL{{\mathcal L}}
\def\R{\mathbb R}

\def\bn{\bm  \nu}

\def\CK{{\mathcal K}}
\def\CL{{\mathcal L}}

\def\CE{{\mathcal E}}
\def\CV{{\mathcal N}}
\def\CS{{\mathcal S}}

\def\sfv{{\mathsf v}}
\def\sfe{{\mathsf e}}

\def\sfm{{\mathsf v_m}}
\def\sfmp{{\mathsf v^{\prime}_m}}
\def\NN{\mathsf N}
\def\RR{\mathsf R}

\def\PP{\mathsf P}

\def\AA{\mathsf A}

\newcommand{\ASSUM}[2]{(\textsf{#1#2})}   \newcommand{\De}[1]{\delta_{e}(#1)} 
\newcommand{\hideit}[1]{}

\newtheorem{theorem}{Theorem}[section]
\newtheorem{lemma}[theorem]{Lemma}

\newtheorem{remark}[theorem]{Remark}

\date{\today}

\title{A two-level method for Mimetic Finite Difference discretizations of elliptic problems}

\author{
Paola F. Antonietti\footnote{MOX, Dipartimento di Matematica, Politecnico di Milano,
  Piazza Leonardo da Vinci 32, 20133, Milano,
  Italy. \texttt{paola.antonietti@polimi.it}},
Marco Verani\footnote{Dipartimento di Matematica, Politecnico di
  Milano, Piazza Leonardo da Vinci 32, 20133, Milano,
  Italy. \texttt{marco.verani@polimi.it}}\ and\ 
Ludmil Zikatanov\footnote{Department of Mathematics, Penn State University, University Park, PA 16802, USA. 
\texttt{ludmil@psu.edu}}~\footnote{Institute of Mathematics and
Informatics, Bulgarian Academy of Sciences, Acad. G. Bonchev str, bl.~8, 1113 Sofia, Bulgaria}
}

\begin{document}

\maketitle 

\begin{abstract}
  We propose and analyze a two-level method for mimetic finite
  difference approximations of second order elliptic boundary value
  problems.  We prove that two-level algorithm is uniformly
  convergent, \emph{i.e.}, the number of iterations needed to achieve
  convergence is uniformly bounded independently of the characteristic
  size of the underling partition. We also show that the resulting
scheme provides a uniform preconditioner with respect to the number
  of degrees of freedom. Numerical results that validate the theory
  are also presented.
\end{abstract}

\section{Introduction}

Thanks to its great flexibility in dealing with very general meshes
and its capability of preserving the fundamental properties of the
underlying physical model, the mimetic finite difference (MFD) method
has been successfully employed, in approximately the last ten years,
to solve a wide range of problems.  Mimetic methods for the
discretization of diffusion problems in mixed form are presented in
\cite{HymanShashkovSteinberg_97,HymanShashkov98,BrezziLipnikovShashkov_05,BrezziLipnikovSimoncini_05,BrezziLipnikovShashkov_06,BrezziLipnikovShashkovSimoncini_07}. The
primal form of the MFD method is introduced and analyzed in
\cite{BrezziBuffaLipnikov_08,BeiraodaVeigaLipnikovManzini_2011}.
Convection--diffusion problems are considered in
\cite{CangianiManziniRusso_09,BeiraodaVeigaDroniouManzini_2011}, while
the problem of modeling flows in porous media is addressed
\cite{LipnikovMoultonSvyatskiy_08}. Mimetic discretizations of linear
elasticity and the Stokes equations are presented in
\cite{BeiraoDaVeiga_2010} and
\cite{BeiraodaVeigaGyryaLipnikovManzini_2009,BeiraodaVeigaLipnikovManzini_2010,BeiraodaVeigaLipnikov_2010},
respectively.  MFD methods have been used in the solution of
Reissner-Mindlin plate equations \cite{BeiraodaVeigaMora_2011}, and
electromagnetic
\cite{BrezziBuffa_2010,LipnikovManziniBrezziBuffa_2011} equations.
Numerical techniques to improve further the capabilities of MFD
discretizations such that \emph{a posteriori} error estimators
\cite{BeiraodaVeiga_2008,BeiraodaVeigaManzini_2008,AntoniettiBeiraoLovadinaVerani_2013}
and post-processing techniques \cite{CangianiManzini_08} have been
also developed.  The application of the MFD method to nonlinear
problems (variational inequalities and quasilinear elliptic equations)
and constrained control problems governed by linear elliptic PDEs is
even more recent, see \cite{AntoniettiBeiraoBigoniVerani_2013} for a
review. More precisely, in
\cite{AntoniettiBeiraoVerani_2013,AntoniettiBeiraoVerani_2013b} a MFD
approximation of the obstacle problem, a paradigmatic example of
variational inequality, is considered. The question whether the MFD
method is well suited for the approximation of optimal control
problems governed by linear elliptic equations and quasilinear
elliptic equations is addressed in \cite{AntoniettiBigoniVerani_2012}
and \cite{AntoniettiBigoniVerani_2012b}, respectively. Recently, in
\cite{BeiraodaVeigaBrezziCangianiManziniMariniRusso_2013}, the mimetic
approach has been recast as the \emph{virtual element method} (VEM),
cf. also \cite{BrezziMarini_2013,BeiraoBrezziMarini_2013}.
Nevertheless, the issue of developing efficient solution techniques
for the (linear) systems of equations arising from MFD discretizations
haas not been addressed right now.  The main difficulty in the
development of optimal multilevel solution methods relies on the
construction of consistent coarsening procedures which are non-trivial
on grids formed by more general polyhedra. We refer to
\cite{2012LashukI_VassilevskiP-aa,2008PasciakJ_VassilevskiP-aa,2008LashukI_VassilevskiP-aa}
for recent works on constructing coarse spaces with approximation
properties in the framework of the agglomeration multigrid
method. Very recently, using the techniques of \cite{CGH:2014,
  ASV:2013}, a multigrid algorithm for Discontinuous Galerkin methods
on polygonal and polyhedral meshes has been analyzed in
\cite{AHSV:2014}.
\\

The aim of this paper is to develop an efficient two-level method for
the solution of the linear systems of equations arising from MFD
discretizations of a second order elliptic boundary value problem.  We
prove that the two-level algorithm that rely on the construction of
suitable prolongation operators between a hierarchy of meshes is
uniformly convergent with respect to the characteristic size of the
underling partition. We also show that the resulting scheme provides a
uniform preconditioner, \emph{i.e.}, the number of Preconditioned
Conjugate Gradient (PCG) iterations needed to achieve convergence up
to a (user-defined) tolerance is uniformly bounded independently of
the number of degrees of freedom.  An important observation is that
for unstructured grids a two-level (and
multilevel) method is optimal if the number of nonzeroes in the coarse
grid matrices is under control. This is important for practical
applications and one of the main features of the method proposed here
is that we modify the coarse grid operator so that the number of
nonzeroes in the corresponding coarse grid matrix is under
control. This in turn complicates the analysis of the preconditioner,
since we need to account for the fact that the bilinear form on the
coarse grid is no longer a restriction
of the fine grid bilinear form.\\

The layout of the paper is as follows.  In Section~\ref{sec:mimetic}
we introduce the model problem and its mimetic finite difference
discretization.  The solvability of the discrete problem is discussed
also in this section and further, spectral bounds of the stiffness matrix arising form MFD discretization are provided in Section \ref{sec:CondNum}. Our two-level preconditioners are
described and analyzed in Section \ref{sec:two-level}. Finally, in
Section \ref{sec:numerics} we present numerical results to validate
the theoretical estimates of the previous sections and to test the
practical performance of our algorithms.

\section{Model problem and its mimetic discretization}\label{sec:mimetic}
Let $\Omega$ be an open, bounded Lipschitz polygon in $\R^2$.  Using
the standard notation for the Sobolev spaces, we consider the
following variational problem: Find $u \in H^1_0(\Omega)$ such that
\begin{equation}\label{erm2}
  \int_{\Omega} {\kappa(\bm{x})\nabla u \cdot \nabla v \dx{\bm{x}}} =  
  \int_{\Omega}{f~v\dx{\bm{x}}},
  \quad \mbox{for all}\quad v \in H^1_0(\Omega) .
\end{equation}
Here, $f\in L^2(\Omega)$ and we assume that the
function $\kappa(\bm{x})$ is a piecewise constant function, bounded and
strictly positive, namely, there exist $\kappa_{\star},\kappa^{\star}>0$ such
that $\kappa_{\star} \leq \kappa(\bm{x}) \leq \kappa^{\star}$.\\ 

We now briefly review the mimetic discretization method for problem
\eqref{erm2} presented in \cite{Brezzi-Buffa-Lipnikov:08} and extended
to arbitrary polynomial order in
\cite{BeiraoLipnikovManzini:XX}. 
In the following, to
avoid the proliferation of constants, by $\lesssim$ we denote an upper
bound that holds up to an unspecified positive constant. Moreover, $(\cdot,\cdot)$ will denote the Euclidean scalar
product in $\ell^2(\R^n)$, and $\| \cdot\|$ its induced norm. Finally,  $(\cdot,\cdot)_X$ and
$\|\cdot\|_X$, will denote the inner product and the norm generated
by a symmetric, positive definite matrix $X$, repsectively.\\

\subsection{Domain partitioning}

We partition $\Omega$ as union of
connected, \emph{convex} polygonal subdomains with non-empty interior. We denote this partition with
$\Omega_H$, and assume it is \emph{conforming},  \emph{i.e.}, the intersection of the closure of two different elements is either empty or is a union of vertices or edges.
Notice that assuming that $\Omega_H$ is made of convex elements is not restrictive
and an algorithm for such decomposition into a small (close to
minimum) number of convex polygons is presented
in~\cite{Chazelle:1984}.
For each polygon $E\in \Omega_H$, $|E|$ denotes its area, $H_E$
denotes its diameter and $H=\max_{E\in\Omega_H}H_E$ is the
characteristic size of the partition $\Omega_h$.  The set of vertices
and edges of the partition is denoted by $\CV_H$ and $\CE_H$,
respectively.  
The vertices and edges of a
particular element $E$ are denoted by $\CV_H^E$ and $\CE_H^E$,
respectively.  
A generic vertex will be denoted by $\sfv$, and a generic edge by $e$. 
We also assume that $\Omega_H$ satisfies the following assumptions, cf. ~\cite{Brezzi-Buffa-Lipnikov:08}.
\begin{assumption}\label{ass:mesh}
  There exists an integer number $N_s$, independent of $H$, such that
  any polygon $E\in\Omega_H$ admits a decomposition into
 at most $N_s$ \emph{shape-regular} triangles;
\end{assumption}
Assumption \ref{ass:mesh} implies the  following
properties which we use later, cf.  \cite{Brezzi-Buffa-Lipnikov:08} for more details.
\begin{enumerate}
\item[\ASSUM{M}{1}] The number of vertices and edges of every
  polygon $E$ of $\Omega_H$ is \emph{uniformly} bounded.
\item[\ASSUM{M}{2}] For every  $E\in \Omega_H$ and for every edge $e$ of $E$, it holds $H_E \lesssim |e|$ and $H_E^2 \lesssim |E|$.
\item[\ASSUM{M}{3}] The following \emph{trace inequality} holds
  \begin{equation*}
\begin{aligned}
& \norm{\psi}_{L^2(e)}^2 \lesssim H_E^{-1} \norm{\psi}_{L^2(E)}^2 + H_E
    \seminorm{\psi}_{H^1(E)}^2
&& \forall\, \psi\in H^1(E).
\end{aligned}
  \end{equation*}
\item[\ASSUM{M}{4}] For every $E$ and for every function $\psi \in
  H^m(E)$, $m \in \mathbb{N}$, there exists a polynomial $\psi_k$ of
  degree at most $k$ on $E$ such that
  \begin{equation*}
    \seminorm{\psi-\psi_k}_{H^l(E)} \lesssim H_E^{m-l} \seminorm{\psi}_{H^m(E)}
  \end{equation*}
  for all integers $0 \le l \le m \le k+1$.
\end{enumerate}

We then consider a fine partition $\Omega_h$ obtained after a uniform
refinement of $\Omega_H$, according to the procedure described in Algorithm
\ref{alg:mesh}. 
\begin{algorithm}[!htbp]
  \caption{Refinement algorithm, see Figure \ref{fig:refinement}.}\label{alg:mesh}
  \begin{algorithmic}[1]
    \ForAll{polygons $E \in \Omega_H$} 
    \State Introduce the point $\xE\in E$ defined as
    \begin{equation*}
      \xE = \frac{1}{n_E} \sum_{\sfv\in \CV_H^E} \x(\sfv) \ ,
    \end{equation*}
    where $n_E$ is the number of vertexes $\sfv$ of $E$, and $\x(\sfv)$ is the position vector of the vertex $\sfv$.
    \State Subdivide $E$ of $\Omega_H$ by connecting each midpoint $\sfm=\sfm(e)$ of each edge $e\in\CE_H^E$ with the point $\xE$, see Figure~\ref{fig:refinement}.
    \EndFor
  \end{algorithmic}
\end{algorithm}
\begin{figure}[!htbp]
  \centering
 \includegraphics[width=0.5\textwidth]{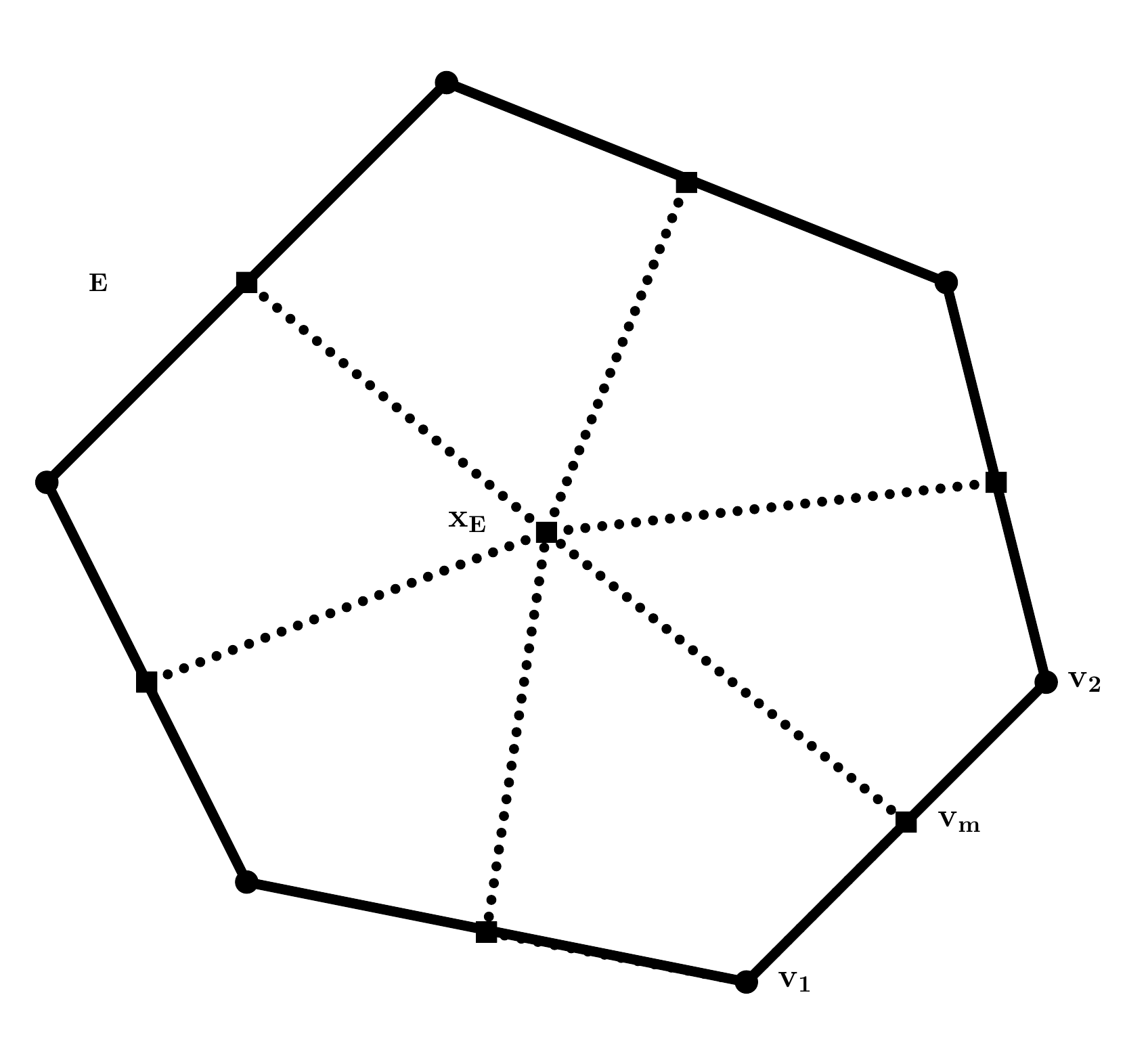}
  \caption{Refinement strategy: a coarse element $E\in \Omega_H$ is subdivided into sub-elements. Circles denote the coarse vertexes in $\CN_H$, while squares refer to additional vertexes in $\CN_h$.}
  \label{fig:refinement}
\end{figure}
Notice that, by construction, the grid $\Omega_h$ automatically satisfies properties $\ASSUM{M}{1}-\ASSUM{M}{4}$. 
As before, the diameter of an element $E\in \Omega_h$ will be denoted by $h_E$, and we set $h=\max_{E\in\Omega_h}h_E$. Accordingly, $\CV_h$ and $\CE_h$ will denote  the sets of vertices
and edges of $\Omega_h$,
respectively. 
We also observe that, according to Algorithm \ref{alg:mesh}, the edge midpoints
$\sfm(e)$ and the points $\xE$ become additional vertexes in the new
mesh ${\Omega}_{h}$, \emph{i.e.},
\begin{equation}\label{eq:vertex-splitting}
{\CV}_{h} = \CV_H \cup \{ \sfm(e) \}_{e\in\CE_H} \cup \{ \xE
\}_{E\in\Omega_H}.
\end{equation}
Finally, we assume that the jumps in $\kappa(x)$ are aligned with the
finest grid and we denote by $\kappa_E$ the coefficient value in the
polygon $E\in \Omega_h$.

\subsection{Mimetic finite difference discretization}\label{sec:bil}
In this section we describe the MDF approximation to problem \eqref{erm2} on the finest grid $\Omega_h$. We begin by
introducing the discrete approximation space $V_h$: any vector $v_h\in V_h$ is given by $v_h=\{v_h(\sfv)\}_{\sfv\in\CV_h}$, where $v_h(\sfv)$ is a
real number associated to the vertex $\sfv\in\CV_h$.  To enforce
boundary conditions, for all nodes of the mesh which lay on the
boundary we set $v_h(\sfv) = 0$. Denoting by $N_h$ the cardinality of $\CV_h$, we have that $V_h \equiv \R^{N_h}$. \\

The mimetic discretization of
problem \eqref{erm2} reads: Find $u_h \in V_h$ such that
\begin{equation}\label{disc2}
\begin{aligned}
& a_h(u_h,v_h) = (f_h,v_h)
&& \forall v_h \in V_h.
\end{aligned}
\end{equation}
where \begin{equation*}
  (f_h,v_h)=\sum_{E\in\Omega_h} \bar{f}|_E\sum_{\sfv_i  \in \CN_h^E} v_h(\sfv_i) \: \o_E^i,
\end{equation*}
with $\bar{f}|_E$ is the average of $f$ over $E$ and $\o_E^i$ are
positive weights such that $\sum_{i} \o_E^i=|E|$. 
The bilinear form $a_h(\cdot,\cdot): V_h\times V_h \to \R$ is 
defined as follows:
\begin{equation*}
\begin{aligned}
&a_h(v_h,w_h)=\sum_{E\in\Omega_h}  a_{h}^E(v_h,w_h)
&& \forall  v_h, w_h \in V_h,
\end{aligned}
\end{equation*}
where, for each $E \in \Omega_h$, $a_h^E(\cdot,\cdot)$ is a symmetric bilinear form that can be constructed in a simple algebraic way, as shown in \cite{Brezzi-Buffa-Lipnikov:08,AntoniettiBeiraoVerani_2013}. We next recall this algebraic expression and use it to show  that \eqref{disc2} is well posed. 
For any $E\in\Omega_h$ let $n_E$  be the number of its vertexes and let  $\AA^{E}_h\in {\mathbb R}^{n_E\times n_E}$
be the symmetric matrix representing the
local bilinear form $a_{h}^E(\cdot,\cdot)$, \emph{i.e.},
\begin{equation*}
\begin{aligned}
&(\AA^{E}_h v_h, w_h)
=a_{h}^E(v_h,w_h) 
&&\forall v_h , w_h \in
V_h.
\end{aligned}
\end{equation*}
We define 
\begin{equation}\label{eq:matrix_M}
  \AA^{E}_h = \frac{1}{\kappa_E\abs{E}} \RR \RR^T + s \: \PP \ ,
\end{equation}
with $s = \textrm{trace}(\frac{1}{\kappa_E\abs{E}} \RR \RR^T) > 0$  a scaling factor. The matrix $\PP$ is defined as
$\PP = {\mathsf I} - \NN (\NN^T \NN)^{-1}\NN^T$, where \begin{equation} \label{eq:matrix_N}
  \NN =
  \begin{pmatrix}
    1    & x_{1}-\bar{x}_{E}  & y_{1}-\bar{y}_E \\
    1   & x_{2}-\bar{x}_{E}  & y_{2}-\bar{y}_E \\
    1    & x_{3}-\bar{x}_{E}  & y_{3}-\bar{y}_E \\
    \vdots &  \vdots       \\
    1    & x_{n_E}-\bar{x}_{E}  & y_{n_E}-\bar{y}_E \\
  \end{pmatrix},
\end{equation}
being 
$\sfv_1=(x_{1},y_{1}),\ldots,\sfv_{n_E}=(x_{n_E},y_{n_E})$ and $(\bar{x}_{E}, \bar{y}_{E})$ the vertexes and the center of mass of $E$, respectively.
The matrix $\RR$ has the following form
\renewcommand{\arraystretch}{1.5}
\begin{align*} \label{eq:matrix_R}
  \RR  &= \frac{\kappa_E}{2}
  \begin{pmatrix}
    0    &  y_2-y_{n_E}  & x_{n_E}-x_{2}\\
    0    &  y_{3}-y_{1}  &   x_{1}-x_{3}\\
    0    &  y_{4}-y_{2}  & x_{2}-x_{4}\\
    \vdots & \vdots &\vdots\\
    0    &  y_{1}-y_{n_E-1}  & x_{n_E-1}-x_{1}
  \end{pmatrix}.
\end{align*}
Note that, by construction, it holds $\AA^E_h \NN = \RR$. \\

We now prove a result which is basic in showing solvability of the discrete problem. 

\begin{lemma}\label{lm:kerM}
The matrix $\AA^E_h$ is positive semidefinite. Moreover, $\AA^{E}_h z = 0$ if and only if $z=(\alpha,\ldots,\alpha)^T$ for
  some $\alpha\in {\mathbb R}$. 
\end{lemma}
\begin{proof}
For any $z \in \mathbb{R}^{n_E}$, using that $\PP^2=\PP$ and $\PP^T=\PP$, we have
\begin{equation}\label{eq:AE}
(\AA^{E}_h z, z ) 
= \frac{1}{\kappa_E\abs{E}} (\RR \RR^T  z,  z )+ s (\PP  z,  z )
    =  \frac{1}{\kappa_E\abs{E}} \| \RR^T  z \|^2 +  s \| \PP  z\|^2 \ge 0.
\end{equation}
  We next show that $\AA^{E}_h z = 0$ if and only if $z=(\alpha,\ldots,\alpha)^T$ for
some $\alpha\in {\mathbb R}$.  One direction of the proof is
easy. Indeed, taking $z=(\alpha,\ldots,\alpha)^T$ for $\alpha\in
\mathbb{R}$, then
  \begin{equation*}
z=\NN\begin{pmatrix}\alpha\\0\\0\end{pmatrix},
  \end{equation*}
and hence
 \begin{equation*}
 \AA^{E}_h z= \AA^{E}_h\NN\begin{pmatrix}\alpha\\0\\0\end{pmatrix}=
    \RR \begin{pmatrix}\alpha\\0\\0\end{pmatrix}=0.
  \end{equation*}
To prove the other direction, let us assume that $\AA^{E}_h z=0$. 
Equation \eqref{eq:AE} clearly implies that $\RR^T  z=0$ and $\PP  z=0$.
  From $\PP z=0$,  we conclude that $z\in
  \operatorname{Range}(\NN)$, and, hence, $z = \NN
  \widetilde{z}$ for some $\widetilde{z}=
(\widetilde z_1, \widetilde z_2, \widetilde z_3)^T\in \mathbb{R}^3$. This yields 
  \begin{equation*}
    \RR
    \widetilde{z}= \AA^{E}_h \NN \widetilde{z}=\AA^{E}_h z=0.
  \end{equation*}
  We now want to show that $(\widetilde z_1, \widetilde z_2,
  \widetilde z_3)^T =(\alpha, 0, 0)^T$ for some $\alpha\in\mathbb{R}$.  
   Indeed, denoting by $\bn_E^e$ the unit normal vector to the edge $e$ pointing outside of $E$, the identity $\RR \widetilde{z}= 0$, shows that
  $(\widetilde z_2, \widetilde z_3)^T \cdot \bn_E^{\sfe_{i}}=0$ for
  $i=1,\ldots,n_E$. As at least two of the normal vectors
  $\{\bn_E^{\sfe_{i}}\}_{i=1}^{n_E}$ are linearly independent, this
  implies that $\widetilde z_2 = \widetilde z_3 = 0$.
  Finally, the proof is concluded by setting $\widetilde z_1=\alpha$, $\widetilde z_2
  = \widetilde z_3 = 0$, and computing $\NN \widetilde{z}$ which yields
  $z = \NN \widetilde{z}= (\alpha,\ldots,\alpha)^T$.
To show that $\AA^{E}_h$ is positive definite on the orthogonal
complement of the constant vectors,
we have to show that
\begin{equation*}
(\AA^{E}_h z, z )>0,
\end{equation*} 
for any $z=(u_1, u_2, u_3)^T$ such that $u_1+u_2+u+3=0$. 
For such $z$ we have 
 $\| \RR^T  z \|\neq 0$ and  $ \| \PP  z\|\neq 0$, and, hence,
\eqref{eq:AE} gives 
\begin{equation*}
(\AA^{E}_h z, z ) = \frac{1}{\kappa_E\abs{E}} \| \RR^T  z \|^2 +  s \| \PP  z\|^2>0,
\end{equation*}
and the proof is complete. 
\end{proof}
As a consequence of 
the second part of Lemma \ref{lm:kerM}, setting $a_{ij}^E=(\AA^E_h)_{ij}$,  we immediately get
\begin{equation*}
a^E_{ii} = 
-\sum_{\substack{j=1\\
j\neq i}}^{n_E}a^E_{ij}. 
\end{equation*}
Denoting 
$u_{h,i}=u_h({\mathsf v}_i)$, $v_{h,i}=v_h({\mathsf v}_i)$ for
${\mathsf v}_i\in \CV_h^E$, 
and,   from this identity we have 
\begin{equation}\label{bilinear-form-1}
  a_h^E(u_h,v_h) = \frac{1}{2}\sum_{i,j=1}^{n_E} (-a^E_{ij}) (u_{h,i}-u_{h,j}) (v_{h,i}-v_{h,j}).
\end{equation}

We now introduce (on $E$) a different bilinear form which is spectrally
equivalent to $a_h^E(\cdot,\cdot)$ but the summation is over fewer
edges. We will denote this new bilinear
form with $a^E(\cdot,\cdot)$ and define it as
\begin{equation}\label{bilinear-form-2}
  a^E(u_h,v_h)  = \sum_{E\in\Omega_h} k_E\sum_{e\in \CE_h^E} \frac{|E|}{h_e^{2}}\De{u_h}\De{v_h},
\end{equation}
where, for every $e\in \CE_h$, we set $\De{v_h}= v_h(\sfv) -
v_h(\sfv^{\prime})$ being $\sfv$ and $\sfv^{\prime}$ the two vertices
of the edge $e$. Based on~\eqref{bilinear-form-2}, we define 
\begin{equation}\label{bilinear-form-2-glob}
  a(u_h,v_h)  = \sum_{E\in\Omega_h} a^E(u_h,v_h).
\end{equation}
We have the following result.
\begin{lemma}\label{lm:bilinear-form-equivalence} 
  The bilinear forms $a(\cdot,\cdot)$ and $a_h(\cdot,\cdot)$ are
  spectrally equivalent with constant depending only on the mesh
  geometry.
\end{lemma}
\begin{proof}
  The spectral equivalence is shown first locally on every $E$. By
  Lemma~\ref{lm:kerM} we have that $A_h^E$ is symmetric
  positive semidefinite with one dimensional kernel and therefore,
  $a_h^E(\cdot,\cdot)$ is a norm on $\mathbb{R}^{n_E}/\mathbb{R}$. Same
  holds for $a^E(\cdot,\cdot)$, namely, it also induces a
  norm on $\mathbb{R}^{n_E}/\mathbb{R}$ (as long as the set of edges
  in $E$ forms a connected graph). It is easily checked that the
  entries $(a_{ij}^E)_{i,j=1}^{n_E}$ and the edge weight
  in~\eqref{bilinear-form-2} are the same order with respect to $h_e$
  and $|E|$. Finally, summing up over all elements $E$ concludes the
  proof.  Clearly, the constants of equivalence depend on
  the number of edges of the polygons, which is assumed to be uniformly bounded (see Assumption~\ref{ass:mesh}).
\end{proof}
Lemma \ref{lm:bilinear-form-equivalence}
implies that we can introduce energy norm on $V_h$ via
$a(\cdot,\cdot)$
\begin{equation}\label{def:discrnorm}
  \normA{v_h}^2  = \sum_{E\in\Omega_h} k_E|E|\sum_{e\in \CE_h^E} \frac{|\De{v_h}|^{2}}{h_e^{2}}.
\end{equation}
Thanks to the Dirichlet boundary conditions, the quantity
  $\normA{\cdot}$ is a norm on $V_h$. For Neumann problem, it this
  will be only a seminorm.  We remark that $\normA{\cdot}$ resembles a
  discrete $H^1(\Omega)$ norm; indeed, the quantity
  $h_h^{-1}\De{v_{h}}$ represents the tangential component of the
  gradient on edges and the scalings with respect to $|E|$ and $h_e$
  give an inner product equivalent to the $H^1(\Omega)$
  on standard conforming finite element spaces.

\subsection{Condition number estimates}
\label{sec:CondNum}
In this section we provide
spectral bounds for the symmetric and positive definite operator  $A_h: V_h \longrightarrow V_h$ 
\begin{equation}\label{eq:MFDoperator}
\begin{aligned}
(A_h u_h,v_h) = a_h(u_h, v_h)  
&& \forall\, u_h, v_h\in V_h.
\end{aligned}
\end{equation}
associated to the MFD bilinear form $a_h(\cdot,\cdot)$.
Instead of working directly with $A_h$, it will be easier to work with the operator
\begin{equation}\label{def:L}
\begin{aligned}
  (A_L u_h,v_h) = a_L(u_h, v_h) 
&& \forall u_h,v_h\in V_h,\end{aligned}
\end{equation}
where the graph-Laplacian bilinear form is defined as
\begin{equation*}
\begin{aligned}
  & a_L(u_h,v_h) 
  =\sum_{E\in\Omega_h} \sum_{e\in \CE_h^E} \De{u_{h}} \De{v_{h}}.
\end{aligned}
\end{equation*}
Defining
\begin{equation*}
\begin{aligned}
& \|v_h\|_{a_L}^2 = a_L(v_h,v_h)  
&& \forall v_h\in V_h,
\end{aligned}
\end{equation*}
the following norm equivalence holds.
\begin{lemma}\label{lem:equiv_A_AL}
For any $v_h \in V_h$ it holds
$$
 \|v_h\|_{a_L} \lesssim  \|v_h\|_{a}\lesssim  \|v_h\|_{a_L},
$$
where the hidden constants depend on $\kappa_{\star}$ and $\kappa^{\star}$.
\end{lemma}

Thanks to Lemma \ref{lm:bilinear-form-equivalence} and Lemma \ref{lem:equiv_A_AL}, 
$A_L$ and $A_h$ are spectrally equivalent, and therefore any spectral bound for the operator $A_L$ also provides a spectral bound for $A_h$.\\

Before stating the main result of this section, we introduce the definition of the Cheeger's constant
associated to the partition $\Omega_h$ (see  \cite{1970CheegerJ-aa}
and 
\cite{1989JerrumM_SinclairA-aa,1984DodziukJ-aa}).  
 Let $\CS$ be a subset of $\CN_h$ and let $\bar{\CS}=\CN_h\setminus \CS$. Denoting by
$\CE(\CS,\bar\CS)$ the set of edges with one endpoint in $\CS$ and
the other in $\bar\CS$, the Cheeger's constant $C_c$ for $\Omega_h$ is
defined as follows 
\begin{equation}\label{eq:cheeger-constant}
C_c = \frac{1}{2\sqrt{m_d}}\min_{\CS\subset\CN_h} \widetilde{C}_c(\CS), \quad 
\widetilde{C}_c(\CS)  = 
\frac{\vert \CE(\CS,\bar\CS)\vert }{\min(\vert
   \CS\vert,\vert\bar\CS\vert)},\quad
m_d=\max_{\sfv \in \CN_h}|\{e\in \CE_h\;\big\vert\; e\supset \sfv\}|
\end{equation}
where $\vert \CS\vert$ and $\vert \CE(\CS,\bar\CS)\vert$ denote the
cardinality of $\CS$ and $\CE(\CS,\bar\CS)$ and $m_d$ is maximum
number of edges connected to a vertex in the graph (the maximum vertex
degree in the graph given by $\Omega_h$). 
The following result provides an estimate of the extremal eigenvalues
of the operator $A_L$ and is a straightforward application of the
results for general graphs given in
\cite[Theorem~2.3]{1984DodziukJ-aa} and \cite[Lemma~3.3]{1989JerrumM_SinclairA-aa}.
\begin{theorem}\label{th:theoremCheeger}
Let $C_c$ be the Cheeger's constant associated with the partition
$\Omega_h$ defined as in \eqref{eq:cheeger-constant}. Then, it holds
  \begin{equation}
\begin{aligned}
&C_c^2 \le \frac{(A_L v_h,v_h)} {(v_h,v_h)}\le m_d && \forall v_h\in V_h.
\end{aligned}
  \end{equation}
\end{theorem}

\begin{remark}
  For (mimetic) finite difference or finite element methods we can obtain a
  quantitative estimate of $C_c$. Indeed, for a typical domain in $d$-spatial dimensions we
  have:
\begin{equation*}
C_c= \frac{1}{2\sqrt{m_d}} \min_{\CS\subset N_h} \frac{|\CE(\CS,\bar{\CS})|}{\min (\vert \CS\vert , \vert \bar{\CS} \vert)}
\gtrsim \frac{h^{1-d}}{h^{-d}} \gtrsim h, \quad\mbox{and}\quad
(A_L v_h,v_h)_{\ell^2}\approx h^{2-d}|v_h|_{H^1(\Omega)}.
\end{equation*}
Although these inequalities might be difficult to prove, they are
  reasonable assumptions about a finite element, or (mimetic) finite
  difference meshes. Evidently, the graph corresponding to a uniform mesh on the square/cube
  satisfies these inequalities. It is then straightforward to see that
  in such case, the lower bound is provided by the usual Poincar\'{e}
  inequality for $H_0^1$ functions. Denoting by $v_h$ the function or the vector representing it
  and rescaling $(v_h,v_h)_{\ell^2}\approx
  h^{-d}\|v_h\|^2_{L^2(\Omega)}$ leads to
\[
\|v_h\|_{L^2(\Omega)}^2 
\lesssim h^{d}(v_h,v_h)_{\ell^2}\lesssim 
h^d C_c^{-2} (A_L v_h,v_h)_{\ell^2} \lesssim 
h^{d-2}h^{2-d}|v_h|^2_{H^1(\Omega)} = |v_h|^2_{H^1(\Omega)}.
\] 
as expected. 
\end{remark}


\section{Two-level preconditioners}
\label{sec:two-level}
In this section we provide the construction of uniform two-level
preconditioners for $a(\cdot,\cdot)$ and prove uniform bound on the
condition number of the preconditioned matrix. Thanks to
Lemma~\ref{lm:bilinear-form-equivalence} a uniform preconditioner for
$a(\cdot,\cdot)$ will also provide a uniform preconditioner for
$a_h(\cdot,\cdot)$ (and viceversa).
We observe that the bilinear form $a(\cdot,\cdot)$ can be written in more compact
form,
\begin{equation}\label{eq:graph-laplace}
\begin{aligned}
&a(u_h,v_h) = \sum_{e\in \CE_h} a_e\De {u_h}\De {v_h}
&& \forall \, u_h, v_h\in V_h,
\end{aligned}
\end{equation}
with $a_e= k_E |E|/h_e^2>0$ for any $e\in\CE_h$, cf. \eqref{bilinear-form-2}.\\

Let $\Omega_H$ be the coarse partition that generated the fine grid
through the refinement procedure described in Algorithm
\ref{alg:mesh} and let
$V_H$ be the coarse MFD space.  We introduce the natural inclusion operator $I_H^h: V_H \rightarrow {V}_{h}$, also
known as the prolongation operator, which characterizes the elements
from $V_H$ as elements in $V_h$. Its action corresponds to an
extension of the coarse grid values to the fine grid vertices by
averaging.  Its definition is the following 
\renewcommand{\arraystretch}{1.5}
\begin{equation*}
\begin{aligned}
\big( I_H^h v_H \big) (\sfv) &=
 v_H(\sfv), &&  \mbox{for all}\quad \sfv\in\CV_H, \\
 \big( I_H^h v_H \big) (\sfm(e)) & = \frac{1}{2} \big( v_H(\sfv) + 
v_H(\sfv') \big),&& \mbox{for all $\sfm(e)$, $e \in \CE_H$} \\ 
\big( I_H^h v_H \big) (\bm{x}_E) &= 
\frac{1}{N_E} \sum_{\sfv\in \CV_H^E} v_H(\sfv)  
&& \mbox{for all $E\in \Omega_H$} 
\end{aligned}
\end{equation*}
\renewcommand{\arraystretch}{1}
 where $x_E$ is defined as in Algorithm \ref{alg:mesh} (see
also Figure~\ref{fig:edgeH1}), and  $\sfm(e)$ is the midpoint of the edge $e\in  \CE_H$.
With an abuse of notation, we still denote by $V_H$ the embedded coarse space obtained from the application of the prolongation operator $I_H^h$. With this notation, we have $V_H \subset V_h$, where each element 
$v_H\in V_H$ is a vector of $\mathbb{R}^{\mathcal{N}_h}$ that is uniquely identified once we fix the values  $v_H(\sfv)$ for all $\sfv\in \mathcal{N}_H$ (the other values result from the action of $I_H^h$).
For future use, we introduce the following two operators that will be useful in the sequel. First, we denote by $\Pi_H:{V}_h \rightarrow V_H$ the standard interpolation
operator, namely, for all $v_h\in {V}_h$, the action $\Pi_{H}v_h$ 
is the element of the coarse space $V_H$ which has the same value as
$v_h$ at the coarse grid vertices, namely,
\begin{equation}\label{eq:Pa}
\Pi_H v_h\in V_H, \quad\mbox{and}\quad 
\big( \Pi_{H} v_h \big) (\sfv) = v_h(\sfv) 
\quad\mbox{for all}\quad \sfv\in\CV_H.
\end{equation}
Finally, we introduce the $\ell^2$ orthogonal projection
$Q_H$ onto the space $V_H$, \emph{i.e.},
\begin{equation*}
(Q_H v_h, v_H)=(v_h,  v_H) \quad \forall v_H \in V_H. 
\end{equation*}


There are several different norms on $V_h$ that we need to use in the
analysis. One is the energy norm $\normA{\cdot}$ that was already
introduced in~\eqref{def:discrnorm}. Further, if $D$ denotes the
diagonal of $A$, then we introduce the $D$-norm $\|v\|_{D}^2=(Dv_h,v_h)$
for all $v_h\in V_h$. This norm is clearly an analogue of a scaled
$L^2$-norm in finite element analysis. A direct computation shows that
\begin{equation}\label{eq:normD}
(Du_h,v_h) = \sum_{\sfv\in \CV_h}
\left(\sum_{e\in \CE_h: e\supset  \sfv} a_e\right)
u_h(\sfv)v_h(\sfv).
\end{equation}
By Schwarz inequality we easily get the bound
\begin{equation}\label{eq:inverse}
\normA{v_h} \le c_D \|v_h\|_D \quad\mbox{for all}\quad v_h\in V_h,
\end{equation}
and the constant $c_D$, by the Gershgorin theorem, can be taken to
equal the maximum  number of nonzeroes per row in $A$. 
On the coarse grid we introduce 
two types of bilinear forms: 
\begin{enumerate}
\item[\emph{i)}] a restriction of the original form
$a(\cdot,\cdot)$ on $V_H$, denoted by $a_H(\cdot,\cdot):V_H\times
V_H\mapsto \mathbb{R}$;  
\item[\emph{ii)}] a sparser approximation to $a_H(\cdot,\cdot)$, which we denote by 
$b_H(\cdot,\cdot):V_H\times V_H\to\mathbb{R}$. 
\end{enumerate}
The latter bilinear
form is build in the same way \eqref{bilinear-form-2} was built from~\eqref{bilinear-form-1}.
The formal definitions are as follows: 
\begin{equation}\label{defBH}
\begin{aligned}
(A_H u_H,v_H) & = a(u_H,v_H),\\
(B_H u_H,v_H) & = b_H(u_H,v_H)=\sum_{e\in \CE_H} a_{e,H}\De {u_H}\De {v_H}
\end{aligned}
\end{equation}
where $a_{e,H}$ is defined later on. The main reason to introduce the approximate bilinear form
$b_H(\cdot,\cdot)$ defined in~\eqref{defBH} is that this form is much
more suitable for computations because the number of nonzeroes in the
matrix representing $B_H$ has less nonzeroes than in the matrix
representing $A_H$.  To see this, and also to show the spectral
equivalence between $A_H$ and $B_H$, we write the restriction of the
operator $A$ on the coarser space in a way that is more suitable for
our analysis.  First, we split the space of edges $\CE_h$ in subsets
of edges on coarse element boundaries and edges interior to the coarse
elements,
\begin{eqnarray*}
\displaystyle \CE_h = \CE_m\cup \left[\cup_{E\in\Omega_H} \CE_{0,E}\right].
\end{eqnarray*}
Here, $e\in \CE_m$ is a subset of $e_H\in \CE_H$, connecting the
mid point of a coarse edge $e_H$ to the vertices of $e_H$.  Thus,
every $e_H\in \CE_H$ gives two edges in $\CE_m$ or we have
\[
\displaystyle \CE_m=\cup_{e_H\in \CE_H} [e_{H,1}\cup e_{H,2}],\quad\mbox{where}\quad
e_{H,1},\ e_{H,2}\in \CE_h.
\]
Further, for every $E\in \Omega_H$, $\CE_{0,E}$ is the set of edges
connecting the mass center of $E$ with the midpoints of its boundary
edges (see Figure~\ref{fig:edgeH1}).
\begin{figure}[!htb]
\centering
\includegraphics*[width=0.45\textwidth]{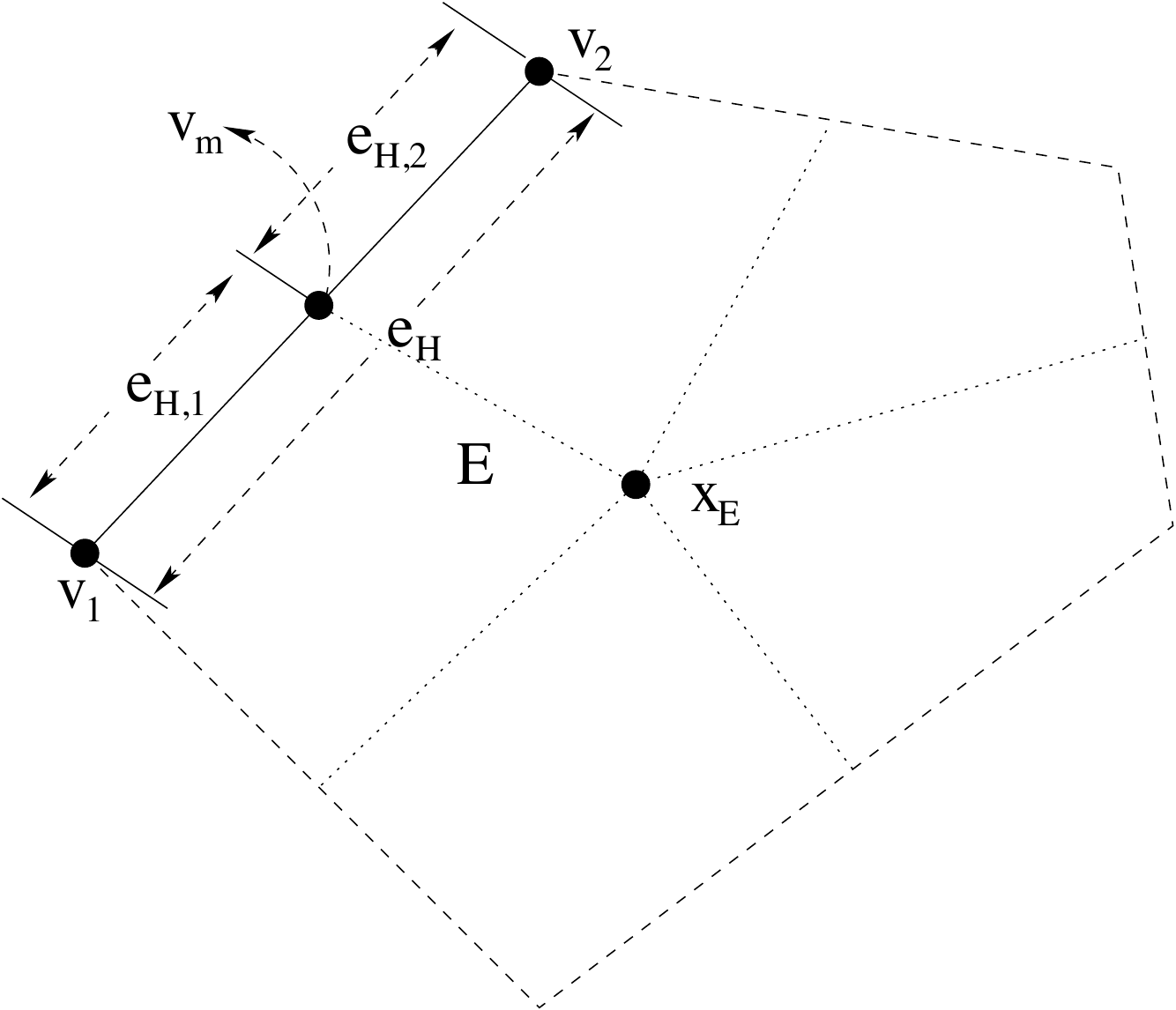}
\caption{A coarse element; boundary and internal edges.} \label{fig:edgeH1}
\end{figure}
With this notation in hand, 
and noticing $\delta_{e_{H,1}} (u_H)=u_H(\sfv_1)-\frac 1
2(u_H(\sfv_1)+ u_H(\sfv_2))=\frac 1 2 (u_H(\sfv_1)-u_H(\sfv_2))$
(analogously for $\delta_{e_{H,2}} $)
we write the restriction of $A$ on
$V_H$ as follows.
\renewcommand{\arraystretch}{1.5}
\begin{equation}\label{eq:coarse-edges}
\begin{array}{rcl}
a_H(u_H,v_H) &=& \displaystyle
\sum_{e_H\in \CE_H} a_{e_{H,1}}\delta_{e_{H,1}} (u_H)\delta_{e_{H,1}} (v_H)
+a_{e_{H,2}}\delta_{e_{H,2}} (u_H)\delta_{e_{H,2}} (v_H)\\
&& \displaystyle +\sum_{E\in \Omega_H}\sum_{e\in \CE_{0,E}} a_e\delta_e(u_H)\delta_e
(v_H)\\
& = & \displaystyle \frac12\sum_{e\in\CE_H}
{a}_{e,H} 
\delta_e (u_H)\delta_e
(v_H) +\sum_{E\in \Omega_H}\sum_{e\in \CE_{0,E}} a_e\delta_e(u_H)\delta_e
(v_H),
\end{array}
\end{equation}
where {
${a}_{e,H}=(a_{e_{H,1}}+a_{e_{H,2}})/2$.
In addition, for any fixed element $E\in \Omega_H$, we obtain
\begin{equation}\label{eq:element-center}
\sum_{e\in \CE_{0,E}} a_e\De {u_H}\De {v_H}
= \sum_{e\in \CE_{0,E}} \frac{1}{n_E}\sum_{e'\in \CE_{0,E}} 
a_e (u_{H}(\sfm)-u_{H}(\sfmp)) (v_H(\sfm)-v_H(\sfmp))
\end{equation}
where we denote by $\sfmp$ 
the midpoint that coincides with one of the  endpoint of
$e^{\prime}\in \mathcal{E}_{0,E}$.
This identity follows from the fact that each of $u_H(\bm{x}_E)$ is an
average of vertex values which is actually equal to the average of
midpoint values for $u_H\in V_H$ and $v_H\in V_H$.  
The (symmetrized) two--grid iteration method computes for any given
initial iterate $u^0$ a two--grid iterate $u^{TG}$ as described in
Algorithm \ref{alg:tg_method} where $R$ denotes a suitable smoothing operator.
\begin{algorithm}[!htbp]
\caption{Two-level algorithm: $u^{TG}\leftarrow u^0 $}
\label{alg:tg_method}
\begin{algorithmic}[1]
\State \emph{Pre-smoothing:} $v=u^0 + R^T(f-A u^0)$; 
\State \emph{Coarse-grid correction:} $e_H=B^{-1}_H Q_H(f-A v),\qquad w= v+e_H$;
\State \emph{Post-smoothing:}  $u^{TG} =w + R(f-Aw)$.
\end{algorithmic}
\end{algorithm}
The error propagation operator $E$ associated with this algorithm 
satisfies the relation
\begin{equation*}
E=(I-RA) (I-B_H^{-1} Q_H A)(I-R^TA).
\end{equation*}
A usual situation is when $E$ is a uniform contraction in
$\normA{\cdot}$-norm. This is definitely the case when $B_H=A_H$.
A proof of this fact follows the same lines as the proof for the case
$B_H\neq A_H$ which we present below. In the case $B_H=A_H$ the
operator $E$ is a contraction because $(I-A_H^{-1} Q_H A)$ is an $A$-orthogonal
projection and therefore non-expansive in $\|\cdot\|_A$-norm and, in
addition, $(I-RA)$ is a contraction in $\|\cdot\|_A$ norm. 

However, when the coarse grid matrix is approximated, \emph{i.e.} we have $B_H\neq
A_H$, then the error propagation operator does not have to be a
contraction and we aim to bound the condition number of the
preconditioned system. In order to do this, we consider the 
explicit form of the two-level MFD preconditioner given by $B^{-1} =
(I-E)A^{-1}$, namely,

\begin{equation}\label{eq:preconditioner}
B^{-1}=\underbrace{R+R^T-R^TAR}_{\widetilde{R}} + (I-AR^T) B_H^{-1} Q_H (I-RA).
\end{equation}
The operator $\widetilde{R} = R+R^T -R^TAR$ is often referred to as the
symmetrization of $R$.

As is well known (see~\cite[pp. 67-68]{2008VassilevskiP-aa} and
\cite{2013HuX_WuS_WuX_XuJ_ZhangC_ZhangS_ZikatanovL-aa}),
if $\|I-RA\|_A<1$ then $\widetilde{R}$ is symmetric positive definite,
and, hence the preconditioner $B$ is symmetric and positive
definite.  Such statement also follows from the canonical form of the
multiplicative preconditioner as given in~\cite[Theorem~3.15,
pp. 68-69]{2008VassilevskiP-aa} and \cite{2008ChoD_XuJ_ZikatanovL-aa}.
\begin {theorem}[Theorem 3.15 in \cite{2008VassilevskiP-aa}]
The following identity holds for the two level preconditioner $B$, given
by~\eqref{eq:preconditioner}
\begin{equation}\label{eq:canonical}
  (Bv,v) = \min_{v_H\in V_H} \left(\|v_H\|^2_{B_H} + \|v-(I-R^TA)v_H\|^2_{\widetilde{R}^{-1}}\right). 
\end{equation}
\end{theorem}
What we will do next is to use this theorem and derive spectral
equivalence results for  $B$ and $A$. 

\subsection{Spectral equivalence results}\label{S:equiv}
In this section we prove that the preconditioner given by the
multiplicative two level MFD algorithm is spectrally equivalent to the
operator $A$.

For the smoother $R$ we assume that it is nonsingular operator and
convergent in $\normA{\cdot}$-norm, that is, 
$$
\normA{I-RA}^2\le
1-\delta_R<1.
$$ 
This implies that the operator $D_R=(R^{-1}+R^{-T}-A)$ is symmetric and
positive definite and also the so called symmetrizations of $R$,
namely $\widetilde{R}=R^TD_RR$ and $\widetilde{R}=RD_RR^T$ are also
symmetric and positive definite. Denoting with $D$ the diagonal of
$A$, we make the following assumptions:
\begin{assumption}\label{as:smoother-1}
  We assume that in the case of nonsymmetric smoother, $R\neq R^T$,
  the following inequality holds with $D_R=(R^{-1}+R^{-T}-A)$ and $D$,
  the diagonal of $A$:
$$(D_Rv,v)\lesssim (D v,v).$$
\end{assumption}
\begin{assumption}\label{as:smoother-2} Let $\widetilde{R}$ be the
  symmetrization of $R$ and $D$ let be the diagonal of $A$. We assume that
$$
(Dv,v)\eqsim (\widetilde{R}^{-1} v,v) 
$$
\end{assumption}

Assumption~\ref{as:smoother-1} obviously holds for a (damped) Jacobi smoother and is
easily verified for Gauss-Seidel or SOR smoother.  For example, in the
case of Gauss-Seidel smoother we have $D_R=D$ and for SOR method with
relaxation parameter $\omega\in (0,2)$ we have
$D_R=\frac{2-\omega}{\omega}D$.
 Assumption~\ref{as:smoother-2} is also a typical assumption in the multigrid methods
 (see~\cite{1985HackbuschW-aa}, \cite{1993BrambleJ-aa}) and is easily verified for Gauss-Seidel method, SOR or
Schwarz smoothers (see~\cite{Zikatanov:2008,2008VassilevskiP-aa}), and also for
polynomial smoothers as well
(see~\cite{2012KrausJ_VassilevskiP_ZikatanovL-aa}).

%


To study the spectral equivalence between the preconditioner defined
by the two level method and $A$ we need some auxiliary results
which are the subject of the next two Lemmas. 
\begin{lemma}\label{lm:basic}
For every $v_h\in V_h$ we have
\begin{equation}
\| v_h - \Pi_H v_h \|^2_{D}\lesssim \normA{v_h}^2.
\end{equation}
\end{lemma} 
\begin{proof}
For $v_h\in V_h$ we have that
\begin{eqnarray}
\big( v_h- \Pi_H v_h \big) (\sfm) &=& 
v_h( \sfm) - \frac{1}{2} \big( v_h(\sfv) + v_h(\sfv') \big) \nonumber \\
 &=&\frac{1}{2} \big( v_h(\sfm) - v_h(\sfv))+ 
\frac{1}{2}(v_h(\sfm) - v_h(\sfv') \big). \label{eq:1}
\end{eqnarray}
Analogously, we obtain 
\begin{eqnarray}
\big( v_h- \Pi_H v_h  \big) (\xE) &=& v_h( \xE) - \frac{1}{n_E}
\sum_{\sfv\in \CN_H^E}
\Pi_H v_h  (\sfv)  \nonumber\\
 &=& 
\sum_{\sfv\in \CN_H^E}
\frac{1}{n_E}(v_h(\x_E)- v_h(\sfv))\nonumber\\ 
 &=& 
\sum_{e\in \mathcal{E}_{0,E}} \frac{1}{n_E}\delta_e (v_h)
\nonumber. \label{eq:2}
 \end{eqnarray}
Next, we use \eqref{eq:1}-\eqref{eq:2} and the definition of $\|\cdot\|_D$
given in~\eqref{eq:normD}.  Splitting the sum over $\sfv\in\CV_h$ in accordance
with~\eqref{eq:vertex-splitting} into: (1) a sum
over the midpoints of coarse edges; and (2) sum over mass centers of
coarse elements; and recalling that $\big(v_h - \Pi_H v_h \big) (\sfv)=0$ for
$\sfv\in\CV_H$ then gives 
\begin{eqnarray}
\| v_h -  \Pi_H v_h  \|^2_D&=& \sum_{\sfv\in \CV_h}
\left(\sum_{e\in \CE_h;\sfv \in e }
  a_e\right)\lbrack(v- \Pi_H v_h )(\sfv)\rbrack^2 \nonumber\\
&=& \frac{1}{2}\sum_{e_H\in\CE_H} (a_{e_{H,1}}+a_{e_{H,1}})(\delta_{e_{H,1}} (v_h)+\delta_{e_{H,2}} (v_h))^2
\nonumber\\
&& + \sum_{E\in\Omega_H} \frac{1}{n_E}\left(\sum_{e'\in \CE_{0,E}}
a_{e'}\right) \sum_{e\in \CE_{0,E}} \lbrack\De {v_h}\rbrack^2 \\
&\lesssim & \normA{v_h}^2. \nonumber
\end{eqnarray}
The proof is complete. 
\end{proof}


\begin{lemma}\label{lm:basic-2} 
The following
  inequalities hold 
\begin{enumerate}[(i)]
\item $\normA{\Pi_H v_h}
\lesssim\normA{v_h}$;
\item $(A v_h, v_h) \le (\widetilde{R}^{-1}v_h, v_h)$;
\item $(R\widetilde{R}^{-1}R^TAv_h,Av_h)\lesssim\normA{v_h}$;
\item $ (B_Hv_H, v_H) \lesssim (A_H v_H, v_H) \lesssim (B_Hv_H, v_H)$.
\end{enumerate}
\end{lemma} 
\begin{proof}
We prove (i) by using the inequality~\eqref{eq:inverse} and the approximation
property proved in Lemma~\ref{lm:basic}
\begin{eqnarray*}
\normA{\Pi_H v_h}&\le& \normA{v_h-\Pi_H v_h} + \normA{v_h} \\
&\lesssim& \|v_h- \Pi_H v_h\|_D + \normA{v_h} \lesssim\normA{v_h}.
\end{eqnarray*}
 
The proof of (ii) follows from the following implications
\begin{eqnarray*}
&&0\le \|(I-RA)v_h\|_A^2\Longrightarrow 0\le ((I-\widetilde{R}A)v_h,v_h)_A \Longrightarrow\\
&& (\widetilde{R}Av_h,Av_h)\le (Av_h,v_h) \Longrightarrow  
(A^{1/2}\widetilde{R}A^{1/2}v_h,v_h)\le (v_h,v_h) \Longrightarrow\\
&& (v_h,v_h)\le (A^{-1/2}\widetilde{R}^{-1}A^{-1/2}v_h,v_h) \Longrightarrow
(Av_h,v_h)\le (\widetilde{R}^{-1}v_h,v_h).
\end{eqnarray*}

Item (iii) follows from Assumption~\ref{as:smoother-1} and 
  its proof is as follows:

\begin{eqnarray*}
(R\widetilde{R}^{-1}R^TAv_h,v_h)_A & = & (D^{-1}_RAv_h,Av_h) 
\le (A^{1/2}D^{-1}A^{1/2}w_h,w_h) \\ 
&\le& \rho(A^{1/2}D^{-1}A^{1/2})(w_h,w_h)  \\
&=&
\rho(D^{-1/2}AD^{-1/2})\|v_h\|^2_A
\lesssim \|v_h\|_A^2.
\end{eqnarray*}

Finally, (iv) follows by using the formulae given in
\eqref{eq:element-center} and \eqref{eq:coarse-edges} and proceeding
as in the proof or Lemma~\ref{lm:bilinear-form-equivalence}. Note that
to prove the spectral equivalence we need to only estimate the second
term on the right side of \eqref{eq:coarse-edges} (or equivalently the
term on the right side of \eqref{eq:element-center}). This is
straightforward using the fact that all norms in a finite dimensional
space are equivalent. 
\end{proof}
In the proof we used~\eqref{eq:element-center}
and~\eqref{eq:coarse-edges} to show that $a_H(\cdot,\cdot)$ and
$b_H(\cdot,\cdot)$ are equivalent. We remark that to achieve that, the
coefficients $a_{e,H}$ of the coarse grid bilinear form
$b_H(\cdot,\cdot)$ in~\eqref{defBH} can be all set
to one. Then the equivalence constants in Lemma~\ref{lm:basic-2} will
depend on the variations in the coefficient $k(x)$. However, other
choices are also possible. One such choice is minimizing the Frobenius
norm of the difference of the local matrices for $b_H(\cdot,\cdot)$
and $a_H(\cdot,\cdot)$. For more details on such approximations that
use the so called edge matrices we refer to~\cite{2007KrausJ-aa}.\\

\begin{remark}
In special cases, the proof of Lemma~\ref{lm:basic-2}(iii) can be done
without using Assumption~\ref{as:smoother-1}.
This is in case the smoother is symmetric  i.e., $R=R^T$ and
$\rho(RA)<1$. Such $R$ could be a symmetrization of a
$A$-norm convergent non-symmetric smoother or just can be a properly scaled symmetric
smoother. Examples, satisfying these assumptions, are the 
symmetric Gauss-Seidel method and the damped Jacobi method with sufficiently
large damping factor (e.g. $R=\frac{1}{\|D^{-1}A\|_{\ell^1}}D^{-1}$). In such cases, we have with
$X=A^{1/2}RA^{1/2}$ and $w_h=A^{1/2}v_h$:
\begin{eqnarray*}
(R\widetilde{R}^{-1}R^TAv_h,v_h)_A & = & ((2I-X)^{-1}Xw_h,w_h) \le
(w_h,w_h)=\|v_h\|_A^2.
\end{eqnarray*}
We used above that $\|X\|=\rho(A^{1/2}RA^{1/2}) = \rho(RA) <1$, or equivalently that $\rho(RA)<1$ and that $\frac{t}{2-t}\in [0,1]$ for  
$t\in [0,1]$. This proves Lemma~\ref{lm:basic-2}(iii) in such special
cases. 
\end{remark}

We are now ready to prove the following uniform preconditioning
result that is obtained using the canonical representation for $B$ given
in~\eqref{eq:canonical}.
\begin{theorem} The condition number of $BA$, $\kappa(BA)$, satisfies
\[
\kappa(BA)\lesssim 1 
\]
\end{theorem}
\begin{proof} In this proof, we use the
  Assumptions~\ref{as:smoother-1}-\ref{as:smoother-2} and
  Lemma~\ref{lm:basic} and Lemma~\ref{lm:basic-2}.
We first show the lower bound. For any $v_h\in V_h$
and $v_H\in V_H$ we have
\begin{eqnarray*}
\|v_h\|_A^2 &\le& 2\|v_h-(I-R^TA)v_H\|_A^2 + 2\|(I-R^TA)v_H\|_A^2\\
&\le & 2\|v_h-(I-R^TA)v_H\|_{\widetilde{R}^{-1}}^2 + 2\|v_H\|_{A}^2
\quad \mbox{[Lemma~\ref{lm:basic-2}(ii)]}\\
& \lesssim &
[\|v_h-(I-R^TA)v_H\|_{\widetilde{R}^{-1}}^2 +
\|v_H\|_{B_H}^2]. \quad \mbox{[Lemma~\ref{lm:basic-2}(iv)]}
\end{eqnarray*}
Taking the minimum over all $v_H\in V_H$ 
and using~\eqref{eq:canonical} then shows that
\[
(A v_h ,v_h) \lesssim (B v_h,v_h).
\]
For the upper bound,  we choose in \eqref{eq:canonical}
$v_H=I_H^hv_h$. We have
\begin{eqnarray*}
(Bv_h,v_h) & = & 
\min_{v_H\in V_H} \left(\|v_H\|^2_{B_H} + \|v_h-(I-R^TA)v_H\|^2_{\widetilde{R}^{-1}}\right)
\\
&\le & \|I_H^hv_h\|^2_{B_H} + \|v_h-I_H^h v_h + R^TAI_H^h v_h\|^2_{\widetilde{R}^{-1}}\\
& \lesssim & \|I_H^hv_h\|_A^2 + \|v_h-I_H^hv_h\|_{\widetilde{R}^{-1}}^2 + \|R^TA I_H^h v_h\|_{\widetilde{R}^{-1}}^2
\quad \mbox{[Lemma~\ref{lm:basic-2}(iv)]}\\
& \lesssim & \|I_H^hv_h\|_A^2 + \|v_h-I_H^hv_h\|_{D}^2 + \|I_H^hv_h\|_{A}^2
\quad \mbox{[Assumption~\ref{as:smoother-2}, Lemma~\ref{lm:basic-2}(iii)]}\\
& \lesssim & \|v_h\|_A^2 + \|v_h\|_A^2 + \|v_h\|_{A}^2
\quad \mbox{[Lemma~\ref{lm:basic}, Lemma~\ref{lm:basic-2}(i)]}\\
& \lesssim & \|v_h\|_A^2,
\end{eqnarray*}

This shows the desired estimate and the  proof is complete. 
\end{proof}

\begin{remark}
  We remark, that a multilevel extension of the results presented here
  is possible via the auxiliary (fictitious) space framework (since
  the bilinear forms are modified). We refer to~\cite{1996XuJ-aa,1992NepomnyaschikhS-aa} and
  \cite[Section~2]{2007HiptmairR_XuJ-aa}) for the relevant techniques
  that allow the extension of the results presented here to the
  multilevel case.
\end{remark}

\section{Numerical Results}
\label{sec:numerics}

We are interested in approximating the solution of the elliptic
problem \eqref{erm2} on the unit square, where the right hand side is
chosen so that the analytical solution is given by
\begin{equation*}
  u(x_1,x_2)=x_1 (x_2 - x_2^2)\exp(x_2)
  \cos\left(\frac{\pi x_1}{2}\right).
\end{equation*}

We start from the initial grids of levels $\CL=1,2$ shown in Figure~
\ref{fig:initial_grids}~(top), that
we denote by $Tria$, $Quad$ and $Hex$ meshes, respectively.  Starting
from these initial grids, we test our two-level solver on a sequence
of finer grids constructed by employing the refinement strategy
described in Section~\ref{sec:two-level}.  More precisely, at each
further step of refinement $\ell=1,2,...$ we consider a uniform
refinement of the grid at the previous level obtained employing the
refinement strategy described in Section~\ref{sec:two-level},
cf. Figure~\ref{fig:initial_grids}~(bottom) for
$\ell=1$, \emph{i.e.}, the meshes obtained after one level of refinement.
\begin{figure}
\centering
\subfigure[Initial level $\CL=1$, fine level $\ell=0$]{
 \includegraphics[width=0.15\textwidth]{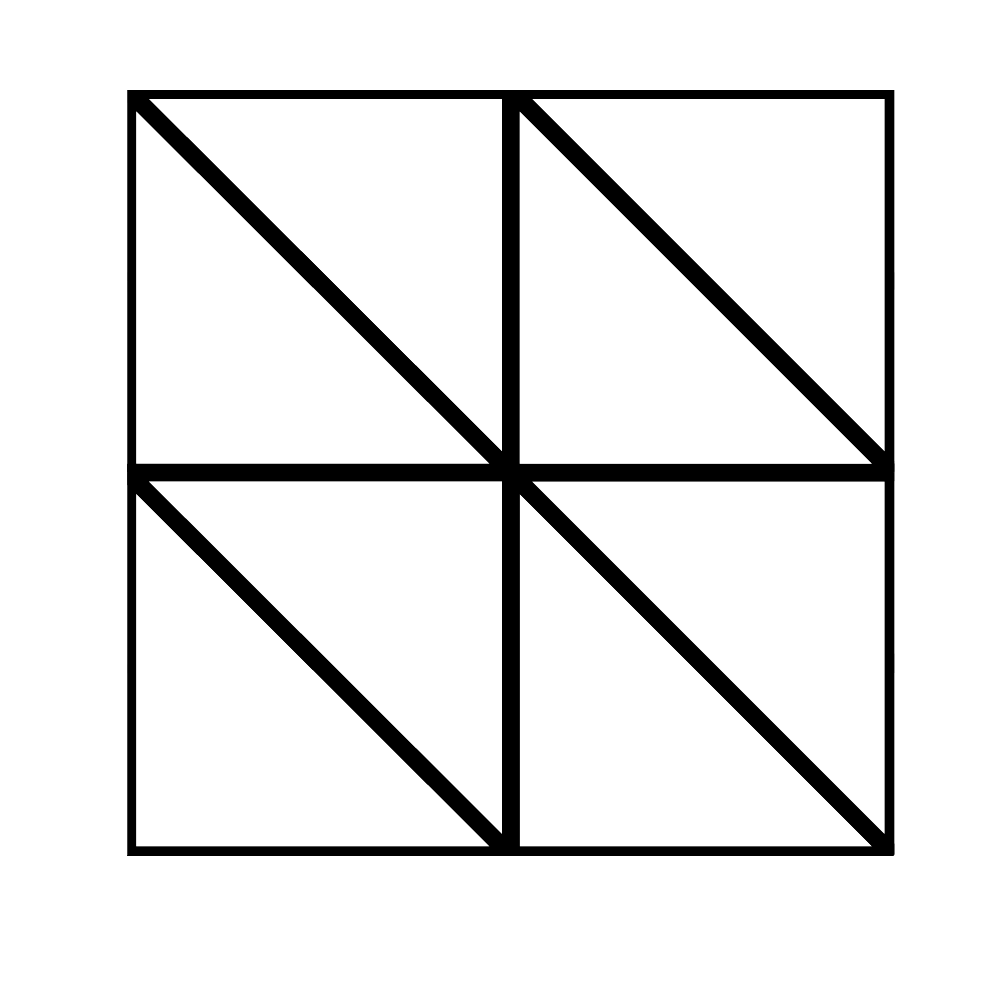}
    \includegraphics[width=0.15\textwidth]{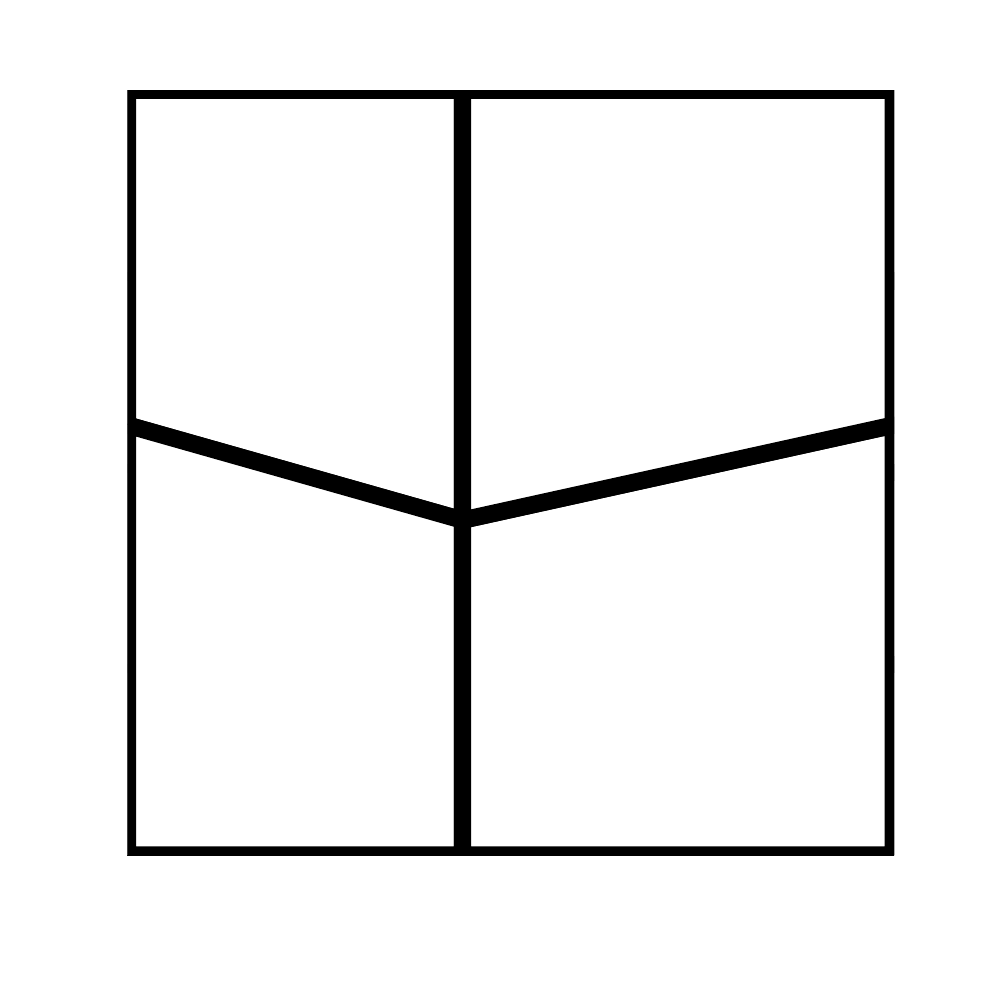}
    \includegraphics[width=0.15\textwidth]{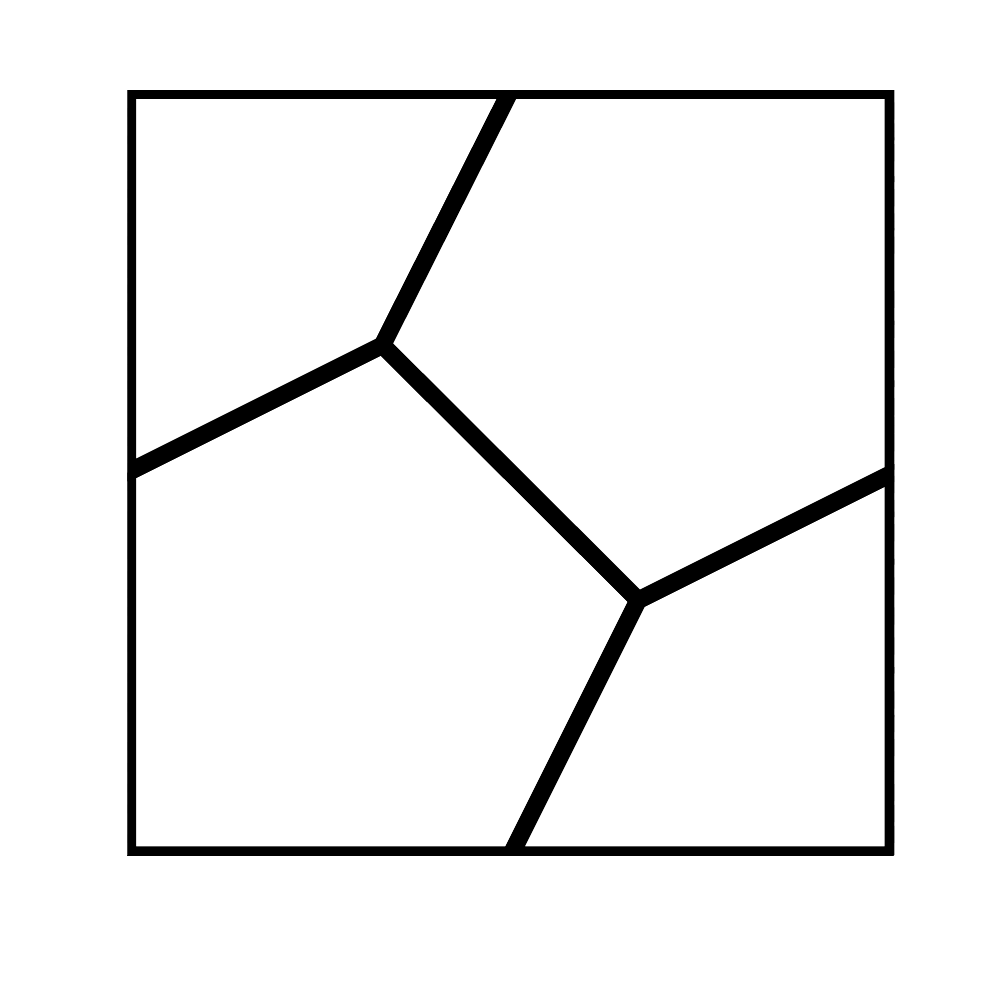}
  }
  \subfigure[Initial level $\CL=2$, fine level $\ell=0$]{
    \includegraphics[width=0.15\textwidth]{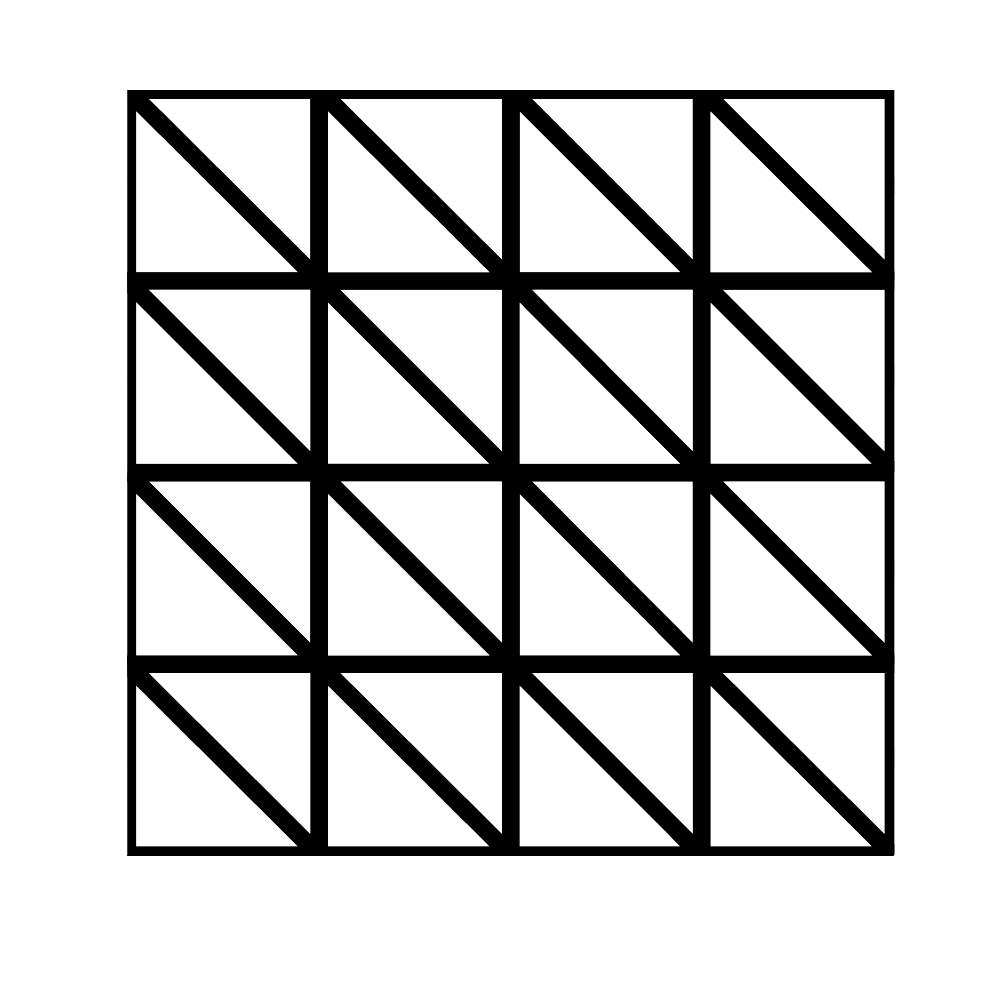}
    \includegraphics[width=0.15\textwidth]{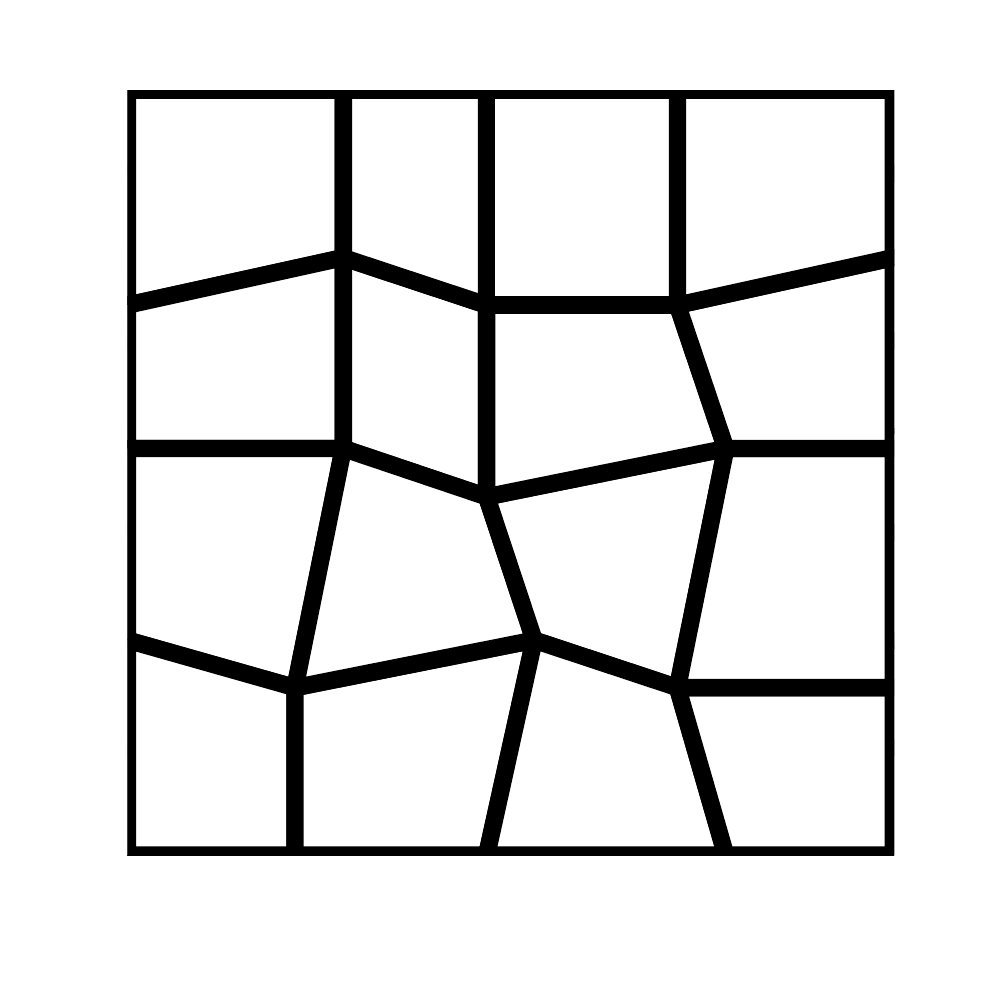}
    \includegraphics[width=0.15\textwidth]{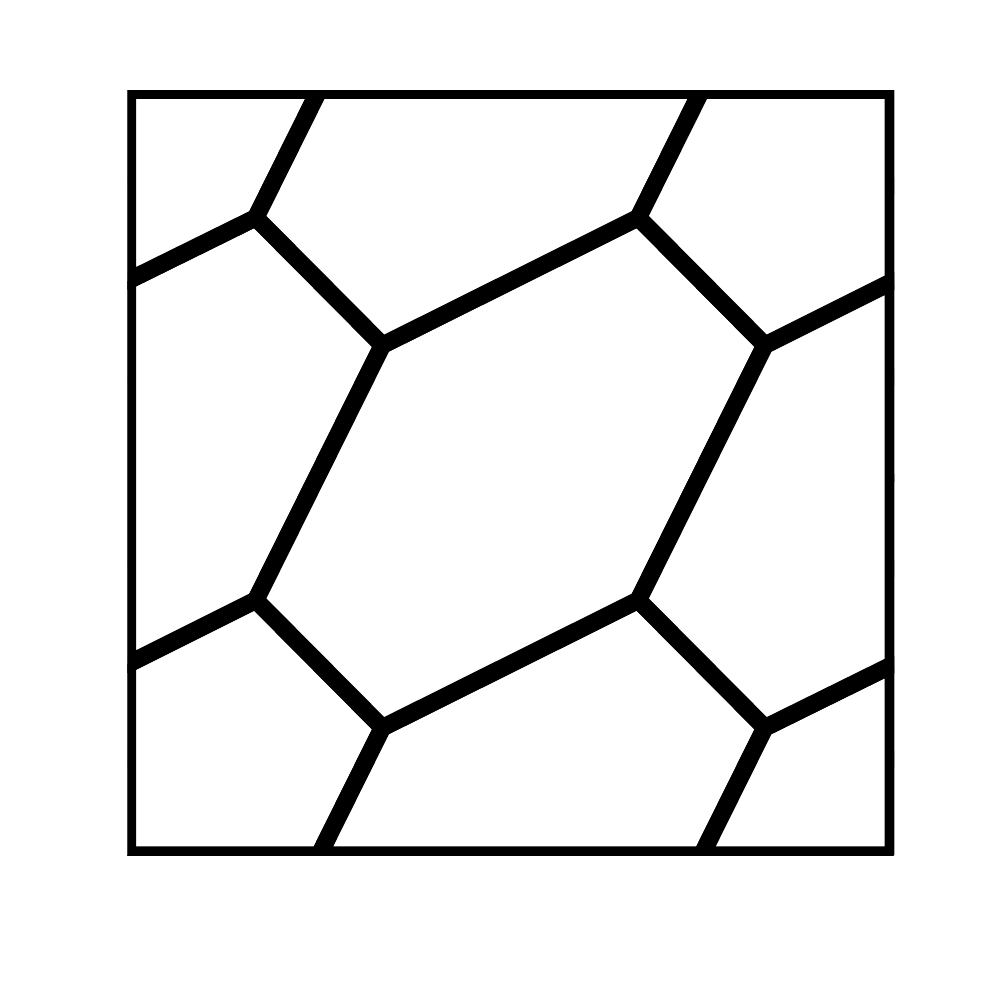}
  }
  \subfigure[Initial level $\CL=1$, fine level $\ell=1$]{
    \includegraphics[width=0.15\textwidth]{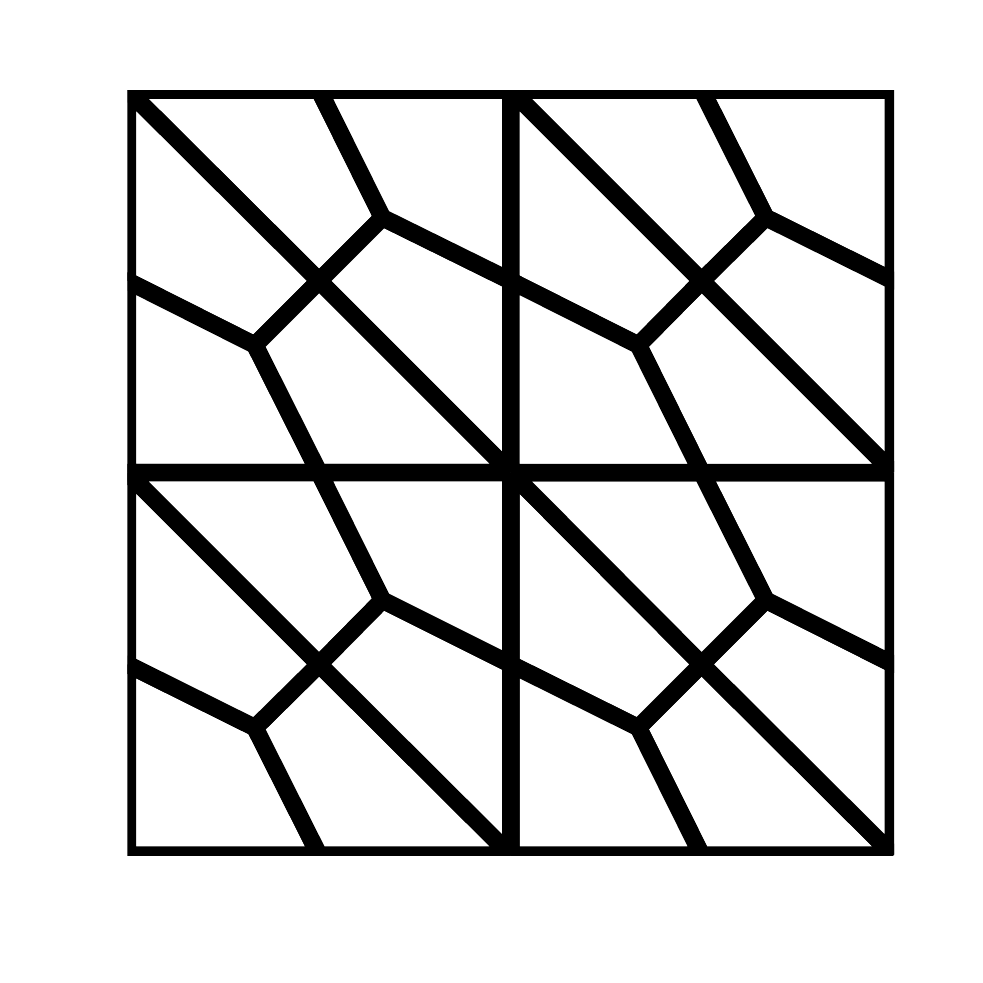}
    \includegraphics[width=0.15\textwidth]{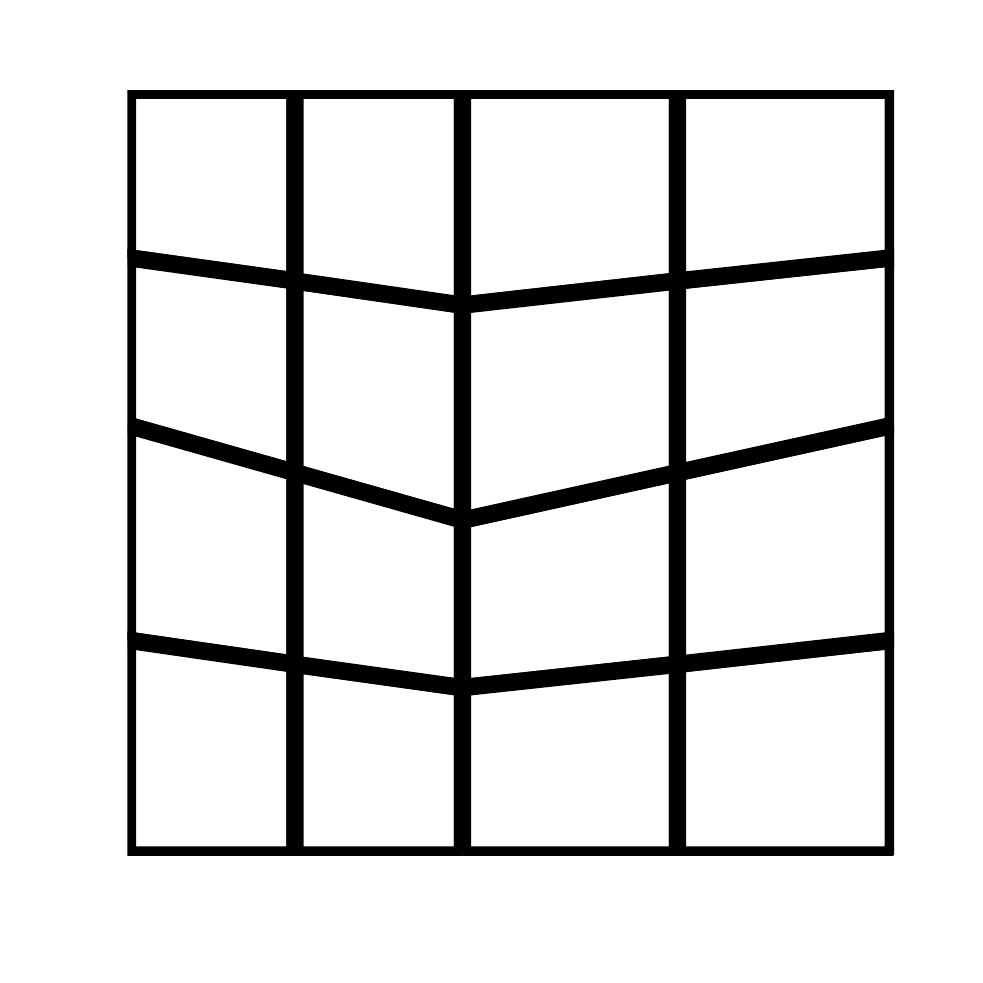}
    \includegraphics[width=0.15\textwidth]{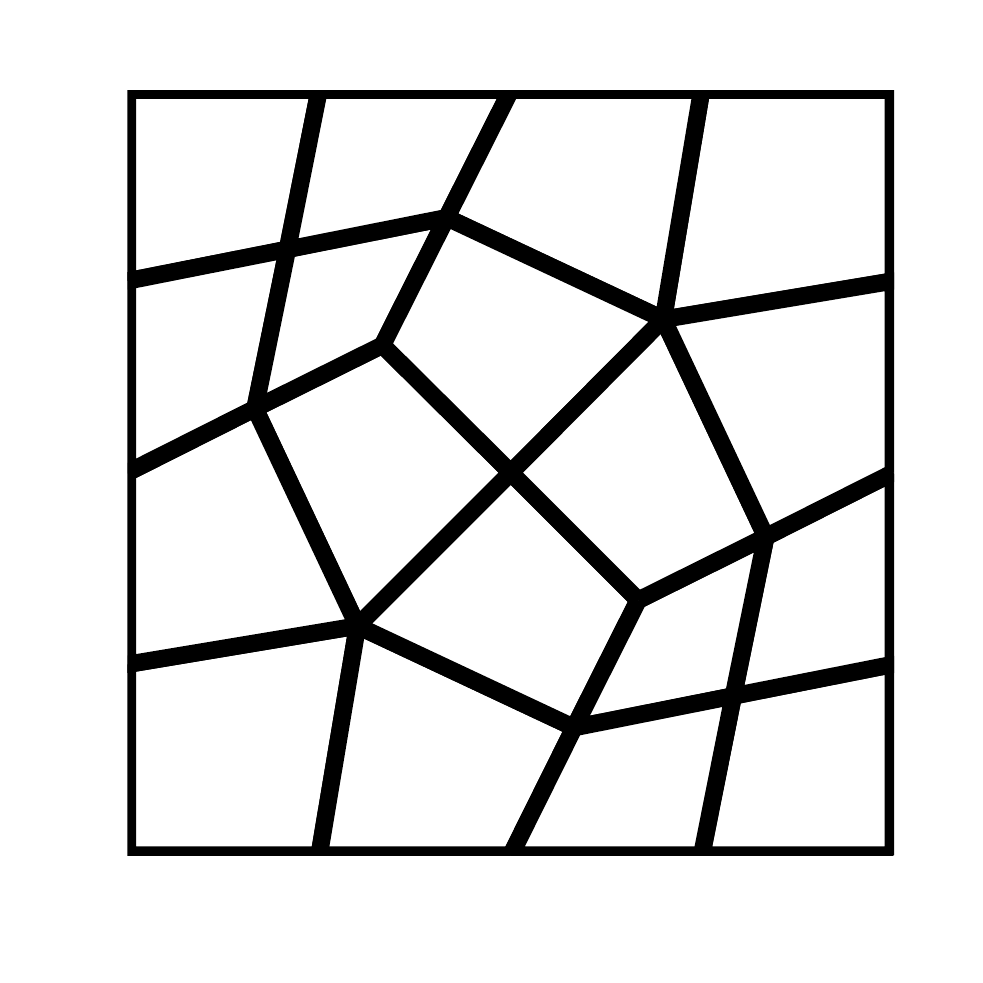}
  }
  \subfigure[Initial level $\CL=2$, fine level $\ell=1$]{
    \includegraphics[width=0.15\textwidth]{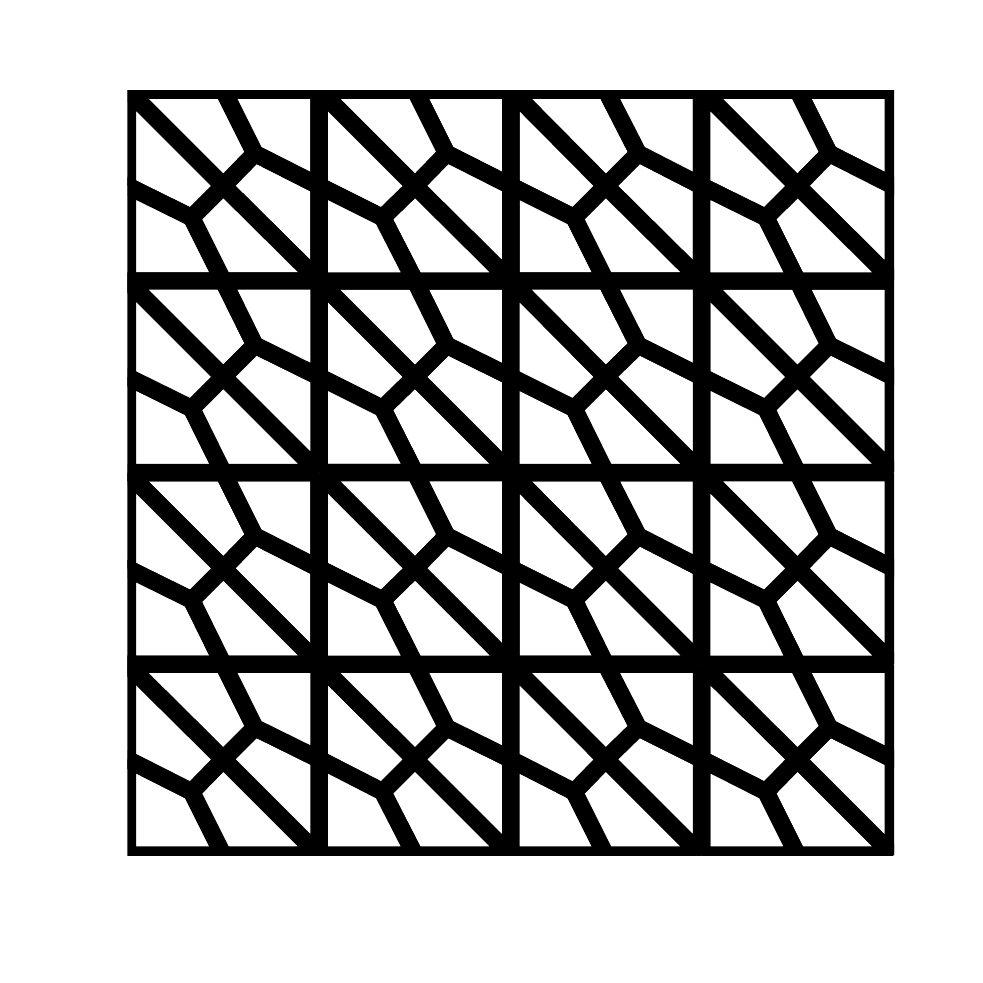}
    \includegraphics[width=0.15\textwidth]{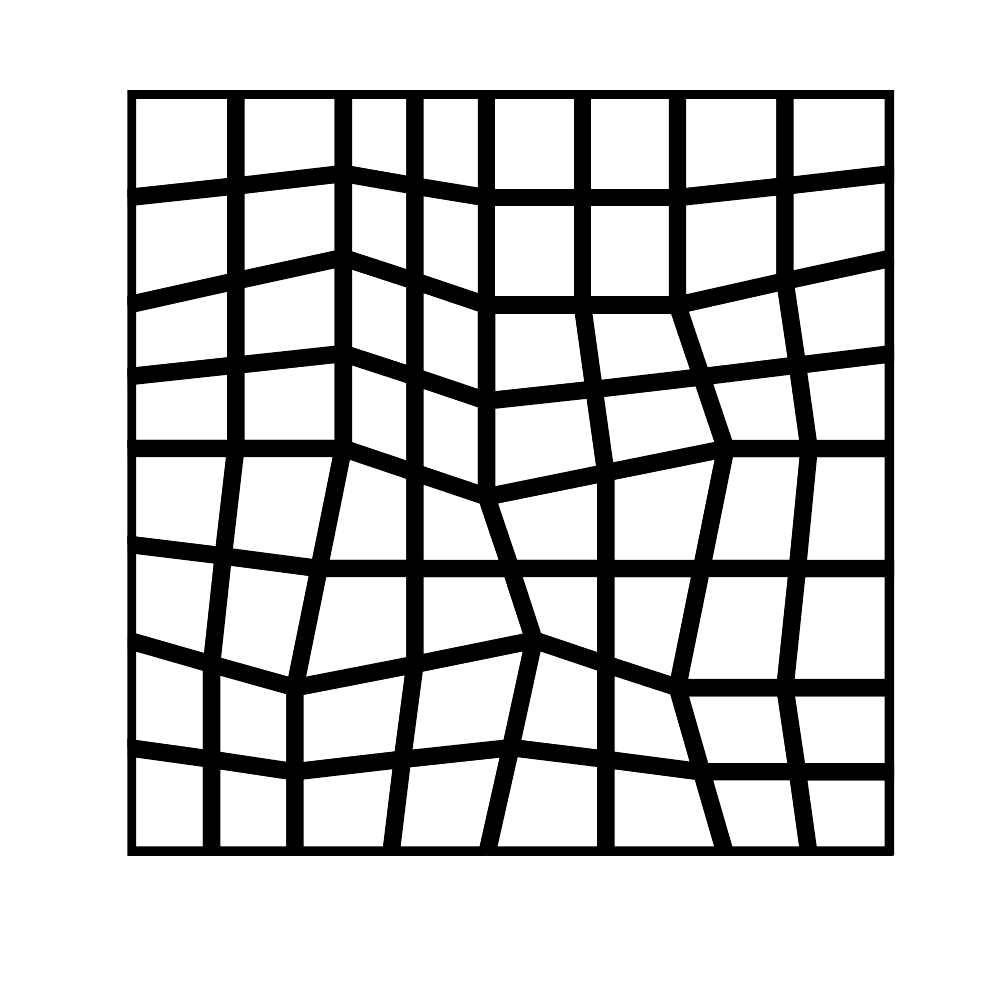}
    \includegraphics[width=0.15\textwidth]{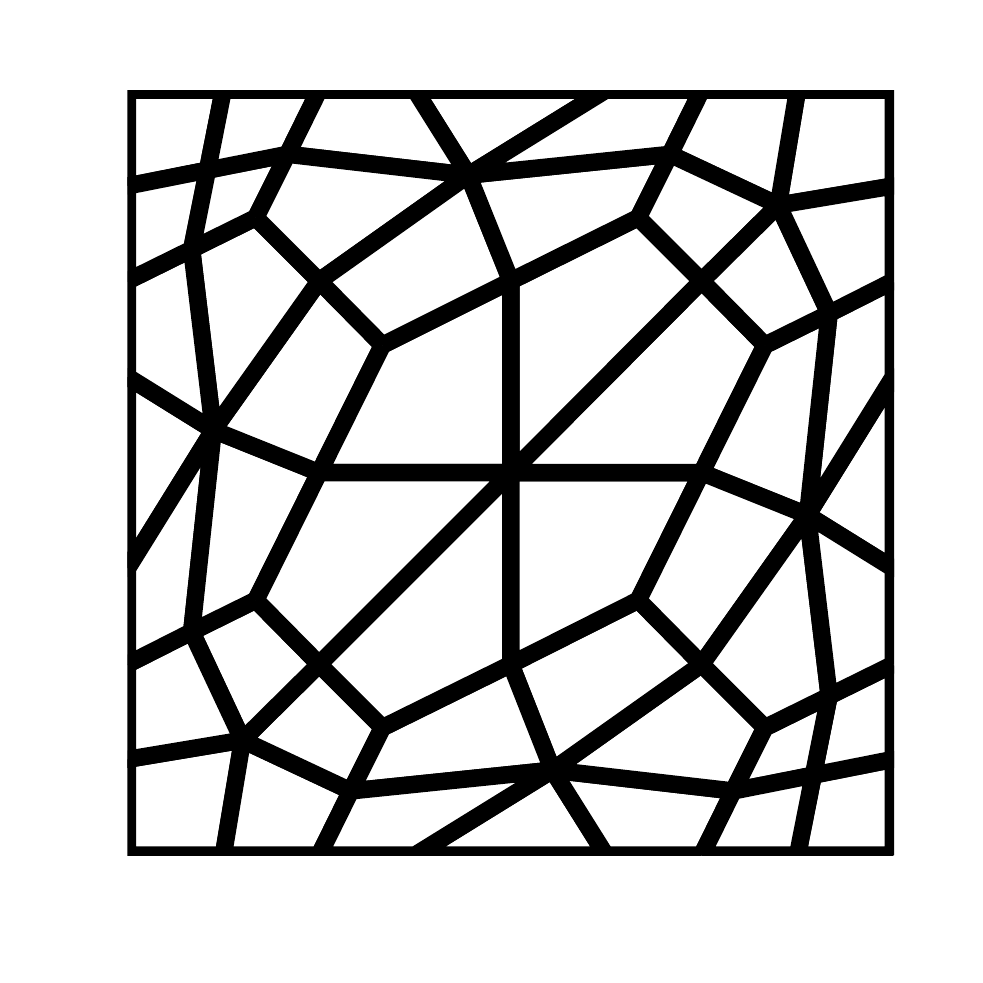}
  }
\caption{Top: $Tria$, $Quad$ and $Hex$ meshes with initial levels $\CL=1$ (left) and  $\CL=2$ (right) and fine level $\ell=0$. Bottom: corresponding grids obtained after a uniform refinement ($\ell=1$) employing the refinement strategy of Section~\ref{sec:two-level}.}
\label{fig:initial_grids}
\end{figure}%
As pre-smoother we employ $\nu$ steps of the Gauss-Seidel iterative algorithm, while a direct solver is employed to solve the coarse problem. All simulations are performed by using the null vector as initial guess, and we use as stopping criterium $\norm{{\bm r}^{(k)}} \leq 10^{-9} \norm{{\bm b}}$, being ${\bm r}^{(k)}$ the residual at the $k$-th iteration,  ${\bm b}$ the right-hand side of the linear system, and $\norm{\cdot}$ the Euclidean norm.\\

\begin{table}[!h]
  \begin{center}
    \begin{tabular}{c |c c c c |c c c c|c c c c}
      & it. & $\rho$ & 
      $\CK(\bm A)$  & rate
      & it. & $\rho$ & 
      $\CK(\bm A)$  & rate
      & it. & $\rho$ & 
      $\CK(\bm A)$  & rate\\
      \hline
      $\ell=1$  
      &  18  & 0.3 & 1.1e+1&  -  
      & 9     & 0.1 &  5.9e+0 & - 
      & 7    & 0.1 &  6.9e+0 & -  \\
      $\ell= 2$  	
      &  13  & 0.2 & 4.9e+1  & 2.2  
      & 8     & 0.1 & 2.6e+1  & 2.1 
      & 8     & 0.1 & 3.2e+1  & 2.2  \\ 
      $\ell= 3$ 	
      &18 & 0.1 & 2.2e+2 & 2.1  
      & 8  & 0.1 & 1.1e+2 & 2.0 
      &10 & 0.1 & 1.4e+2 & 2.1\\  
      $\ell= 4$  	
      & 22& 0.4 & 9.2e+2 & 2.1  
      &9   & 0.1 & 4.2e+2 & 2.0 
      &11 & 0.1 & 6.2e+2 & 2.1\\  
      $\ell= 5$  	
      &23 & 0.4 & 3.9e+3 & 2.0  
      & 9  & 0.1 & 1.7e+3 & 2.0 
      &12 & 0.2 & 1.1e+4 & 2.1 \\
      \hline
      & \multicolumn{4}{c|}{$Tria$ grids}
      & \multicolumn{4}{c|}{$Quad$ grids}
      & \multicolumn{4}{c}{$Hex$ grids}
    \end{tabular}
  \end{center}
  \caption{Iteration counts of the two-level algorithm and computed convergence factor $\rho$ for different fine refinement level $\ell$ starting from the initial grids of in Figure~\ref{fig:initial_grids} with
$\CL=1$. For completeness, the condition number of the stifness matrix $\CK(\bm A)$ and its growth rate are also reported.  Number of pre-smoothing steps $\nu=2$.}
  \label{tab:iterations_L=1_2smoothing}
\end{table}
In Table~\ref{tab:iterations_L=1_2smoothing} we report,  starting from the initial grids shown in Figure~\ref{fig:initial_grids} with $\ell=0$, and $\CL=1$, the iteration counts of our two-level algorithm when varying the fine refinement level $\ell$. This set of experiments has been obtained with $\nu=2$ pre-smoothing steps. We clearly observe that  our solver seems to be robust as the mesh size goes to zero: indeed the iteration counts are almost independent of the size of the problem. 
In Table~\ref{tab:iterations_L=1_2smoothing} we also show the computed convergence factor 
\begin{equation}\label{eq:conv_factor}
  \rho  = \textrm{exp}\left(\frac{1}{n} \log \frac{\norm{{\bm r}^{(n)}}}{\norm{{\bm r}^{(0)}}}\right) ,
\end{equation}
where $n$ is the number of iterations needed to achieve convergence.  Finally, for completeness, we have also computed the condition number of the stiffness matrix $\kappa(\bm A)$ as well as its growth rate (cf. Table~\ref{tab:iterations_L=1_2smoothing}). As expected, we can clearly observe that the condition number increases quadratically as the mesh is refined.  \\

\begin{table}[!htbp]
  \begin{center}
    \begin{tabular}{c |c c c c |c c c c|c c c c}
      & it. & $\rho$ & 
      $\CK(\bm A)$  & rate
      & it. & $\rho$ & 
      $\CK(\bm A)$  & rate
      & it. & $\rho$ & 
      $\CK(\bm A)$  & rate\\
      \hline
      $\ell=1$  
      &  16  & 0.3 & 4.3e+1&  -  
      & 8     & 0.1 & 2.7e+1 & - 
      & 7    & 0.1 &  1.3e+1 & -  \\
      $\ell= 2$  	
      &  14  & 0.2 & 2.0e+1  & 2.2  
      & 9     & 0.1 & 1.1e+2  & 2.1 
      & 14     & 0.2 & 6.5e+1  & 2.4  \\ 
      $\ell= 3$ 	
      &17 & 0.2 & 8.6e+2 & 2.1  
      &10  & 0.1 &4.6e+2 & 2.0 
      &18 & 0.3 & 3.3e+2 & 2.4\\  
      $\ell= 4$  	
      & 21& 0.4 & 3.7e+3 & 2.1  
      &10 & 0.1& 1.9e+3 & 2.0 
      &22 & 0.4 & 2.1e+3 & 2.6\\  
      \hline
      & \multicolumn{4}{c|}{$Tria$ grids}
      & \multicolumn{4}{c|}{$Quad$ grids}
      & \multicolumn{4}{c}{$Hex$ grids}
    \end{tabular}
  \end{center}
  \caption{Iteration counts of the two-level algorithm and computed
    convergence factor $\rho$ for different fine refinement levels
    $\ell$ starting from the coarse grids of in
    Figure~\ref{fig:initial_grids} with $\CL=2$. For completeness,
    the condition number of the stifness matrix $\CK(\bm A)$ and its growth rate are also reported.  Number of pre-smoothing steps $\nu=2$.}
  \label{tab:iterations_L=2_2smoothing}
\end{table}
We have repeated the same set of experiments starting from the initial
grids depicted in Figure~\ref{fig:initial_grids} with
$\CL=2$ and $\ell=0$. The computed results are reported in
Table~\ref{tab:iterations_L=2_2smoothing}. Notice that, in this case,
on $Hex$-type grids the condition number seems to grows slightly
faster than expected.

\noindent
Next, we address the influence of the number of smoothing steps of the
performance of our two-level solver. In Table~
\ref{tab:quad_iterations_smooothing_steps} we report the iteration
counts when increasing the number of pre-smoothing steps
$\nu=3,4,5$. The results shown in
Table~\ref{tab:quad_iterations_smooothing_steps} have been obtained
starting from the initial grids of
Figure~\ref{fig:initial_grids} with $\CL=1$ and  $\ell=0$; the corresponding
ones obtained with the initial grids of
Figure~\ref{fig:initial_grids}, $\CL=2$ and $\ell=0$ are completely
analogous and are not reported here, for the sake of brevity. From the
iteration counts reported in
Table~\ref{tab:quad_iterations_smooothing_steps} we can conclude that
\emph{(i)} in all the cases considered, our two-level method is robust
as the mesh size is refined; \emph{(ii)} as expected, the performance
of the algorithm improves as the number of smoothing steps increases.
\begin{table}[!htb]
  \begin{center}
    \begin{tabular}{c |r r r|r r r|r r r}
      &\multicolumn{1}{c}{$\nu=3$}
      &\multicolumn{1}{c}{$\nu=4$}
      &\multicolumn{1}{c|}{$\nu=5$}
      &\multicolumn{1}{c}{$\nu=3$}
      &\multicolumn{1}{c}{$\nu=4$}
      &\multicolumn{1}{c|}{$\nu=5$}
      &\multicolumn{1}{c}{$\nu=3$}
      &\multicolumn{1}{c}{$\nu=4$}
      &\multicolumn{1}{c}{$\nu=5$}\\
      \hline
      $\ell=1$  &
      11 & 9 & 8 &
      7  & 6  &  5&
      6  & 6 & 5   \\
      $\ell=2$ & 
      10 & 9 & 8 &
      7  & 6 &  6 &
      7  &  6 & 6\\
      $\ell=3$ &
      11 & 11 & 9&
      7  & 6 &  6  &
      8  &  7 & 7 \\
      $\ell=4$ &
      15 & 12 & 10&
      7   & 6 &  6 &
      8  &  8 & 7 \\
      $\ell=5$ &
      16 & 13 & 11&
      7   & 6  &  6 &
      9  &  8 & 7 \\
      \hline
      & \multicolumn{3}{c|}{$Tria$ grids}
      & \multicolumn{3}{c|}{$Quad$ grids}
      & \multicolumn{3}{c}{$Hex$ grids}
    \end{tabular}
  \end{center}
  \caption{Iteration counts as a function of the number of pre-smoothing steps $\nu=3,4,5$ and for different fine refinement levels $\ell$ starting from the initial grids of  Figure~\ref{fig:initial_grids}, $\CL=1$.  }
  \label{tab:quad_iterations_smooothing_steps}
\end{table}
\newline

Finally, we demonstrate numerically that our scheme also provides a
uniform preconditioner, that is the number of PCG iterations needed to
achieve convergence up to a (user-defined) tolerance is uniformly
bounded independently of the number of degrees of freedom whenever CG
is accelerated by the preconditioner described in
Section~\ref{sec:two-level}.  In Table \ref{tab:prec-PCG} we report
the PCG iteration counts as a function of the number of the fine level
$\ell=1,2,3,4,5$ starting from the initial grids shown in
Figure~\ref{fig:initial_grids} ($\CL=1,2$, $\ell=0$) for $Hex$-type
grids. For completeness, we also report the computed convergence
factor $\rho$ (second and fifth columns) and the correspondindg CG
iteration counts needed to solve the unpreconditioned system (third
and sixth columns).  It is clear that employing our preconditioner
leads to a uniformly bounded number of iterations (independent of the
characteristic size of the underling partition). On the other hand,
the iteration counts needed to solve the unpreconditioned systems
grows linearly as the mesh size goes to zero.
\begin{table}[!htb]
\centering
\begin{tabular}{c |c cc|c cc}
&\multicolumn{1}{c}{PCG it.}
&\multicolumn{1}{c}{$\rho$}
&\multicolumn{1}{c|}{CG it.}
&\multicolumn{1}{c}{PCG it.}
&\multicolumn{1}{c}{$\rho$}
&\multicolumn{1}{c}{CG it.}\\
\hline
$\ell=1$  & 10 & 0.25& 19  & 11 & 0.27 & 30\\
$\ell=2$  &12 &  0.29&42  & 10 & 0.25 & 66\\
$\ell=3$  &10 &  0.23&92  & 10& 0.23 & 133\\
$\ell=4$  &10 &  0.23& 210  & 10 &0.24 &  324\\
$\ell=5$  &10 &  0.25& 533  &  - & - & - \\
\hline
& \multicolumn{3}{c|}{$\CL=1$}
& \multicolumn{3}{c}{$\CL=2$}\\
\end{tabular}
\caption{PCG iteration counts and computed convergence factor $\rho$ as a function of the number of level $\ell$ starting from the initial grids of  Figure~\ref{fig:initial_grids}, $\CL=1,2$, $Hex$ grids. For comparison, the CG iteration counts needed to solve the unpreconditioned systems are also reported. }
\label{tab:prec-PCG}
\end{table}

\section{Conclusions}
We have proposed and analyzed a two level preconditioner for  mimetic
finite difference discretizations of elliptic
equations.  
Our preconditioner use inexact coarse grid solver (non-inherited coarse
grid bilinear form) and results in a optimal
method with sparser coarse grid operators. We proved that
the condition number of the preconditioned system is uniformly
bounded. We also implemented the preconditioner and verified
numerically the theoretical results. 

\section{Acknowledgements}
Part of this work was completed while the third author was visiting
MOX at Politecnico di Milano in 2013.  Thanks go to the MOX for the
hospitality and support.  
The work of the first and second author has been partially 
founded by the 2013 GNCS project \emph{``Aspetti emergenti nello studio di strategie adattative per problemi differenziali''}.
The research of the third author was supported in part by NSF  
DMS-1217142, NSF DMS-1418843,  and Lawrence Livermore National Laboratory through subcontract B603526.

\FloatBarrier

\bibliographystyle{abbrv}
\bibliography{MFD}
\end{document}